\documentclass{amsart}

\usepackage[T1]{fontenc}
\usepackage[latin10]{inputenc}
\usepackage{amssymb, amsmath, amsthm}
\usepackage{hyperref, amsrefs}
\usepackage{enumerate}
\usepackage{verbatim}

\def\dd{\mathcal D}
\def\sd{\sum_{I \in \dd}}
\def\sqd{\sum_{Q \in \dd}}
\def\sumeps{\sum_{\varepsilon \in \{0, 1\}^p \setminus \{0\}}}
\def\fd{\mathbf f}
\def\Fd{\mathbf F}
\def\gd{\mathbf g}
\def\Gd{\mathbf G}
\def\fb{\mathcal B}
\def\ud{\mathbf U}
\def\vd{\mathbf V}

\def\C{\mathbb C}
\def\R{\mathbb R}
\def\E{\mathbb E}
\renewcommand{\S}{\mathcal{S}}
\def\eps{\varepsilon}

\def\lt{L^2}
\def\ltrp{L^2(\mathbb{R}^p)}
\def\ltw{L^2(W)}
\def\ltwi{L^2(W^{-1})}

\newtheorem{theorem}{Theorem}[section]
\newtheorem{definition}[theorem]{Definition}
\newtheorem{lemma}[theorem]{Lemma}
\newtheorem{corollary}[theorem]{Corollary}
\newtheorem{remark}[theorem]{Remark}
\newtheorem{proposition}[theorem]{Proposition}

\title[Sharp bounds and \(T1\) Theorem for matrix kernels and matrix weights]{Sharp bounds and \(T1\) theorem for Calder\'{o}n-Zygmund operators with matrix kernel on matrix weighted spaces }

\author{Sandra Pott}
\address{Centre for Mathematical Sciences, University of Lund, P.O. Box 118, SE-221 00 Lund, Sweden}
\email{sandra@maths.lth.se}
\author{Andrei Stoica}
\address{Centre for Mathematical Sciences, University of Lund, P.O. Box 118, SE-221 00 Lund, Sweden}
\email{andrei@maths.lth.se}


\begin{document}

\begin{abstract}
For a matrix $A_2$ weight $W$ on $\R^p$, we introduce a new notion of $W$-Calder\'on-Zygmund matrix kernels, following earlier work in \cite{Is15}. We state and prove a $T1$ theorem for such operators and give
a representation theorem in terms of dyadic $W$-Haar shifts and paraproducts, in the spirit of \cite{Hy12a}. Finally, by means of a Bellman function argument,
we give sharp bounds for such operators in terms of bounds for weighted matrix martingale transforms and paraproducts.
\end{abstract}	
	\maketitle

\section{Introduction}
One of the most important topics in Harmonic Analysis is the study of weighted norm inequalities for Calder\'{o}n-Zygmund operators. This goes back to the 1970's with the now classical works of R.A. Hunt, B. Muckenhoupt and R.L. Wheeden \cite{HuMuWh73} and R.R. Coifman and C. Fefferman \cite{CoFe74}. While the equivalence between the boundedness 
of Calder\'{o}n-Zygmund singular integral operators on the weighted space \(L^p(w)\) and the membership of the scalar weight \(w\) to the so-called \(A_p\) class has been shown in the above papers, the relation between the operator norm of a Calder\'{o}n-Zygmund operator and the \(A_p\) characteristic of the weight, \([w]_{A_p},\) remained open for quite some time. In the case \(p=2,\) the so-called  \(A_2\) conjecture (now theorem) stated that this dependence is linear in \([w]_{A_2}\). It took roughly the first decade of the 2000's to go from the J. Wittwer's proof of the conjecture for the dyadic martingale transforms (see \cite{Wi00}) to the proof for general Calder\'{o}n-Zygmund operators by T. Hyt\"{o}nen \cite{Hy12a}. Important contributions have been brought by S. Hukovic, S. Treil and A. Volberg \cite{HukTrVo00}, who proved the conjecture for the dyadic square function, by S. Petermichl and A. Volberg \cite{PeVo02}, who showed the linear dependence for the Beurling-Ahlfors transform, and by S. Petermichl \cite{Pe07}, who proved it for the Hilbert and Riesz transforms. 

There are several ways to extend the classical (scalar) theory of Calder\'{o}n-Zygmund operators. One of these extensions assumes that the functions are \(\mathbb{C}^d\)-valued, but the kernel of the Calder\'{o}n-Zygmund operators is still scalar-valued. Inspired by applications to multivariate stationary stochastic processes, an analogue notion of the Muckenhoupt \(A_p\) weights, the matrix \(A_p\) weights, have been introduced by S. Treil and A. Volberg (see \cite{TrVo97}). M. Goldberg \cite{Go03}, F. Nazarov and S. Treil \cite{hunt} and A. Volberg \cite{Vo97}, showed that certain Calder\'{o}n-Zygmund operators are bounded on 
\(L^p(W)\) with \(1<p<\infty\) if \(W\) is a matrix \(A_p\) weight. However, a corresponding \(A_2\) conjecture in the matrix setting is still open and currently attracting much interest.
 In \cite{BiPeWi14}, K. Bickel, S. Petermichl and B. Wick modified a scalar argument to show that the dependence of the norm of the martingale and Hilbert transform on the \(A_2\) characteristic of the weight \(W\) is at most \([W]_{A_2}^{3/2}\log{[W]_{A_2}}.\) The authors of the current paper proved in \cite{PoSt15} that the dependence of the norm of all Calder\'{o}n-Zygmund operators with cancellation
  on \([W]_{A_2}\) is the same as the one for the matrix martingale transform, thus reducing the \(A_2\) conjecture for all such operators to the proof of the linear bound for the latter operator. In a recent paper, F. Nazarov, S. Petermichl, S. Treil and A. Volberg \cite{NaPeTrVo17} showed that for all Calder\'{o}n-Zygmund operators the dependence is no worse than \([W]_{A_2}^{3/2}\). The first sharp results in this direction have recently been obtained by T. Hyt\"{o}nen, S. Petermichl and A. Volberg \cite{HyPeVo17}, and by Isralowitz, Kwon, and the first author in \cite{IsKwPo14}. Here, the linear (in \([W]_{A_2}\)) upper bound was proved for the matrix-weighted square function $S_W$, which can be seen as an average of the matrix martingale transforms we consider, and for the matrix-weighted maximal function $M'_W$, respectively.

The theory can be generalised even further by taking Calder\'{o}n-Zygmund operators with matrix-valued kernel. These operators appear naturally in geometric function theory, multivariate prediction theory or in the study of Toeplitz operators. On a matrix-weighted space, such an operator associated to a matrix kernel can now no longer be considered on its own, but has to be considered together with the matrix weight.

For the case of dyadic paraproducts, this was done in \cite{IsKwPo14}, where  the correct version of matrix weighted BMO spaces was introduced, and 
boundedness of the dyadic paraproducts \(\Pi_B\) on \(L^p(W)\) was characterised by means of a matrix weighted Carleson embedding theorem (the necessary definitions and results are given in the following section). Building on this work and inspired by the proof of the scalar \(A_2\) theorem of T. Hyt\"{o}nen, J. Isralowitz \cite{Is15} then introduced a notion of  \(W\)-Calder\'{o}n-Zygmund operators
and proved a matrix weighted \(T1\) theorem for these operators on
$L^p(W)$, $1 < p < \infty$.

In the current paper, we only consider the case $p=2$, and we state and prove a $T1$-Theorem on $L^2(W)$. Compared to the results in \cite{Is15}, our aim is two-fold. First, and most importantly,
 we give sharp bounds in the $T1$ theorem, in terms of bounds for matrix martingale transforms 
and bounds for matrix paraproducts on matrix-weighted spaces. 
The key for this result is the notion of $W$-dyadic Haar shifts, which we introduce, and sharp bounds for these shifts (Theorem \ref{mainshift}). The proof of the sharp bound relies on a Bellman function argument
for matrix weights in \cite{PoSt15}, 
which was originally inspired by Treil's work in the scalar case
\cite{Tr11}. The sharp bounds for matrix-weighted matrix martingale transforms and paraproducts are conjectured to be linear in the $A_2$-characteristic, but this has so far been out of reach. However, the bounds we can prove
come close to the best known bounds for scalar kernels and matrix weights, which were recently proved by \cite{NaPeTrVo17} by means of convex Lerner-type operators.

Note that these results do not easily give rise to to sharp bounds for $p \neq 2$, since no extrapolation method is known in the matrix setting, and we do not have any generalisation of our Bellman function 
approach for $p \neq 2$.

Our second aim is to give a more general definition of W-Calder\'on-Zygmund kernel than \cite{Is15}, which has in particular local decay and smoothness conditions. Moreover, our $T1$ theorem works with a natural weak boundedness condition
which is easily seen to be necessary.

The paper is organised as follows. In Section \ref{two}, we state the necessary definitions, some background, and the main results, Theorems \ref{mainthm} and \ref{mainshift}. In Section \ref{three}, we prove
the T1 theorem \ref{mainthm}, up the bound from Theorem \ref{mainshift}. Finally, in Section \ref{four}, we prove Theorem \ref{mainshift} with a Bellman function argument.

%While also very interesting in its own right, there has been little progress in this area until the last few years, when J. Isralowitz, H.K. Kwon and S. Pott \cite{IsKwPo14} considered a matrix weighted \(T1\) theorem for a class of matrix kernelled Calder\'{o}n-Zygmund operators (the so-called \(W\)-Calder\'{o}n-Zygmund operators). In that paper, the authors also introduced the correct version of matrix weighted BMO spaces and characterised the boundedness of the dyadic paraproducts \(\Pi_B\) on \(L^p(W)\) by means of a matrix weighted Carleson embedding theorem (the necessary definitions and results are given in the following section). 
%
%In the current paper, we state and prove a $T1$-Theorem for $W$-Calder\'{o}n-
%The idea of writing this paper appeared in the late stages of preparation of \cite{PoSt15}. The main result here is showing that the dependence of the norm of any \(W\)-Calder\'{o}n-Zygmund operator on the \(A_2\) characteristic \([W]_{A_2}\) is the same as for the matrix martingale transform. Being able to show that the latter obeys a linear (in \([W]_{A_2}\)) bound would thus prove the full \(A_2\) conjecture for this class of Calder\'{o}n-Zygmund operators, but so far this has been out of reach. Adapting the operator-valued representation theorem of T. H\"{a}nninen and T. Hyt\"{o}nen \cite{HaHy16}, we decompose a \(W\)-Calder\'{o}n-Zygmund operator into \(W\)-Haar shifts and paraproducts.  [ADD MORE HERE]
%

\section{Definitions and statement of the main results}	\label{two}

\subsection{Matrix \(A_2\) and \(A_\infty\) weights}

For \(p, d \geq 1\), the non-weighted Lebesgue space \(\ltrp\) consists of all measurable functions \(f:\mathbb{R}^p \to \mathbb{C}^d\) with norm
\[\|f\|_{\ltrp} := \Big(\int_{\mathbb{R^p}}\|f(t)\|^2_{\mathbb{C}^d}\, \mathrm{d}t \Big)^{1/2} < \infty.\] 
We will also use the space \(C^1_c(\mathbb{R}^p)\) of compactly supported, continuously differentiable functions \(f:\mathbb{R}^p \to \mathbb{C}^d\).

Let \(\mathcal{M}_d(\mathbb{C})\) be the space of \(d \times d\) complex matrices. A matrix weight on \(\mathbb{R}^p\) is a measurable locally integrable function \(W: \mathbb{R}^p \to \mathcal{M}_d(\mathbb{C}) \) whose values are almost everywhere positive definite. We define \(\ltw\) to be the space of measurable functions \(f:\mathbb{R}^p \to \mathbb{C}^d\) with norm 
\[\|f\|^2_{\ltw} = \int_{\mathbb{R}^p} \|W^{1/2}(t)f(t)\|^2_{\mathbb{C}^d}\, \mathrm{d}t = \int_{\mathbb{R}^p} \langle W(t)f(t), f(t) \rangle \, \mathrm{d}t < \infty.\]
It is well-known that the dual of \(\ltw\) can be identified with \(\ltwi\), where the duality between these two spaces is given by the standard inner product.

We say that a matrix weight \(W\) satisfies the matrix \(A_2\) Muckenhoupt condition, if 
\begin{equation}\label{muckenhoupt}
[W]_{A_2} := \sup_{I} \bigg \| \Big(\frac{1}{|I|} \int_{I} W(t)\, \mathrm{d}t \Big)^{1/2} \Big(\frac{1}{|I|} \int_{I} W^{-1}(t)\, \mathrm{d}t \Big)^{1/2} \bigg \| < \infty,
\end{equation}
where the supremum is taken over all cubes \(I \subset \mathbb{R}^p\), and \(\|\cdot\|\) denotes the norm of the matrix acting on \(\mathbb{C}^d\). 
The number \([W]_{A_2}\) is called the \(A_2\) characteristic of the weight \(W\). We say that a matrix weight \(W\) satisfies the dyadic matrix Muckenhoupt condition \(A_2^d\) on \(\mathbb{R}^p\) or \(\mathbb{R}\), if \eqref{muckenhoupt} is satisfied, but with the supremum being now taken only over dyadic cubes or intervals, respectively  (see \cite{TrVo97}).

Furthermore, we say that a matrix weight \(W\) satisfies the weak matrix \(A_\infty\) condition, if the scalar weights $\langle W e, e \rangle$ are $A_\infty$-weights in the scalar sense, uniformly for all $e \in \C^d$.
We define the weak matrix $A_\infty$ characteristic $[W]_{A_\infty}$ by 
$$
    [W]_{A_\infty} := \sup_{e \in \C^d, e \neq 0} \Big[ \big  \langle  W e, e \big  \rangle \Big]_{A_\infty, FW} 
    $$
where $[\cdot]_{A_\infty, FW}$ denotes the Fuji-Wilson version of the scalar $A_\infty$-constant, see e.g. \cite{Wil08}.

To our knowledge, so far sharp bounds in terms of matrix $A_2$ and weak matrix $A_\infty$ characteristic, which correspond to the known sharp bounds in the scalar case, are known for only two of the important operators in this setting. Both will play an important role in the following.
One is the matrix dyadic weighted square function,
$$
   S_W f (t) = \left( \sum_{\varepsilon \in \{0, 1\}^p \setminus \{0\}}   \sum_{I \in \dd^0} \frac{\chi_I(t)}{|I|} \left\| \langle W \rangle_I^{1/2} \langle f, h_I^{\eps} \rangle \right\|^2 \right)^{1/2}
$$
(see also \cite{PePo02}, \cite{BiPeWi14}). For the notation on dyadic cubes and Haar coefficients, see Subsection \ref{subsec:dy} below. The sharp bound
\begin{equation}   \label{eq:sq}
   \| S_W\|_{\ltw \to \lt} \le C_{p,d} [W]_{A_2}^{1/2}  [W]_{A_\infty}^{1/2}
\end{equation}
was very recently proved in \cite{HyPeVo17}. A natural conjecture, which currently remains open, is the following lower bound corresponding to the scalar weight case,
\begin{equation}   \label{eq:conj}
   \|f \|_{\ltw}  \le C_{p,d}   [W^{-1}]_{A_\infty}^{1/2}   \| S_W f\|_{\ltrp} .
\end{equation}
The best known lower bound appears in \cite{BiPeWi14}, following an argument from \cite{TrVo97}. It is
\begin{equation}   \label{eq:lower}
   \|f \|_{\ltw}  \le C_{p,d}   [W]_{A_2}^{1/2}  \log([W]_{A_2})^{1/2}  \| S_W f\|_{\ltrp} .
\end{equation}

The second operator, for which a sharp bound is known, is the matrix-weighted maximal function $M_W'$,
$$
    M_W'f(t) =  \sup_{t \in I, I \text{ cube }} \frac{1}{|I|} \int_I \left\|  \langle W \rangle_I^{1/2}   f(x) \right \| dx.
$$
It was shown in \cite{IsKwPo14} that the bound
\begin{equation}  \label{eq:maxa2}
   \|M_W' \|_{\ltw \to \lt}  \le C_{p,d} [W]_{A_2}
\end{equation}
holds, which is the best possible in terms of the $A_2$-characteristic. However, one easily sees that the proof in  \cite{IsKwPo14} gives actually a slightly better bound, namely
\begin{equation}  \label{eq:max}
   \|M_W' \|_{\ltw \to \lt}  \le C_{p,d} [W]_{A_2}^{1/2}  [W]_{A_\infty}^{1/2}.
\end{equation}
This can be seen by noticing that the $\epsilon$ chosen in the proof  of (\ref{eq:maxa2} )
for use in the reverse H\"older property can actually be chosen as $[W]_{A_\infty}^{-1}$, according to the sharp reverse H\"older result for $A_\infty$-weights proved in 
\cite{HyPer13}.

\subsection{Dyadic setting}   \label{subsec:dy}
The standard system of dyadic cubes in \(\mathbb{R}^p\) is
\[\dd^0:=\bigcup_{j \in \mathbb{Z}} \dd^0_j, \quad {\rm where}\  \dd^0_j:=\{2^{-j}([0,1)^p +k): k \in \mathbb{Z}^p\}.\]
Given a  sequence \(\omega=\{\omega_i\}_{i \in \mathbb{Z}} \in (\{0,1\}^p)^{\mathbb{Z}}\), the dyadic system \(\dd^{\omega}\) in \(\mathbb{R}^p\) is defined by 
\(\dd^{\omega}:= \{I \dotplus \omega: I \in \dd^0\},\) where the translated dyadic cube \(I \dotplus \omega\) is defined as
\[I \dotplus \omega := I + \sum_{i: 2^{-i} < \ell(I)} 2^{-i} \omega_i.\]
When the particular choice of \(\omega\) is not important, we will use the notation \(\dd\) for a generic dyadic system. We equip the set \(\Omega:=(\{0,1\}^p)^{\mathbb{Z}}\) with the canonical product probability measure \(\mathbb{P}_{\Omega}\) which makes the coordinates \(\omega_i\) independent and identically distributed: the probability of each coordinate \(\omega_i\) taking any of the values in \(\{0,1\}^p\) is \(2^{-p}\). We denote by \(\mathbb{E}_{\Omega}\) the expectation over the random variables \(\omega_i, i \in \mathbb{Z}\).

Let us introduce a few useful notations. For a cube \(I \in \dd\), let \(\ell(I)\) and \(|I|\) denote its side length and volume, respectively. Let
\[\dd_n(I):=\{J \in \dd: J \subseteq I, \ell(J) = 2^{-n} \ell(I)\}\]
be the collection of \(n\)-th generation children of \(I\). For any dyadic cube \(I \in \dd\), we will denote its parent by \(\tilde{I}\).

Any system of dyadic cubes \(\dd\) has an associated function system, the Haar functions. When \(p=1,\) any dyadic interval \(I\) has two Haar functions associated with it:

\[h^0_I = |I|^{-1/2} \chi_{I}, \qquad h^1_I=|I|^{-1/2}(\chi_{I^+}-\chi_{I^-}),\]
where \(\chi_I\) is the characteristic function of the interval \(I\), and \(I^+\) and \(I^-\) are the left and right children of I, respectively.

If \(p>1,\) the Haar functions associated to a cube \(I = I_1 \times \cdots \times I_p\) are 
\[h_I^{\varepsilon}(x) = h_{ I_1 \times \cdots I_p}^{(\varepsilon_1, \cdots, \varepsilon_p )}(x_1, \cdots, x_p) = \prod_{i=1}^{p} h_{I_i}^{\varepsilon_i}(x_i),\]
where \(\varepsilon \in \{0, 1\}^p\).

It is well known that the Haar functions \(h_I^{\varepsilon},\) with \(I \in \dd\) and \(\varepsilon \in \{0, 1\}^p \setminus \{0\},\) form an orthogonal basis of \(\ltrp\). Hence any function \(f \in \ltrp\) admits the orthogonal expansion
\[f=\sd \sum_{\varepsilon \in \{0, 1\}^p \setminus \{0\}} \langle f,h_I^{\varepsilon} \rangle h_I^{\varepsilon}.
\]
We denote the average of a locally integrable function \(f\) on the cube \(I\) by \(\langle f \rangle_I :=|I|^{-1}\int_I f(t)\, \mathrm{d}t\).

\subsection{Matrix martingale transforms and dyadic $W$-Haar shifts}

Let \(W\) be a matrix weight. For a sequence of \(d \times d\) matrices \(\sigma=\{\sigma_I\}_{I \in \dd}\), we introduce the notation 
$$
\|\sigma\|_{\infty, W} = \sup_{I \in \dd} \big \| \langle W \rangle_I ^{1/2} \sigma_I \langle W \rangle_I ^{-1/2} \big \| .
$$
For a sequence \(\sigma\) such that \( \|\sigma\|_{\infty, W} < \infty\), we define the martingale transform operator \(T_{\sigma}\) by
\[T_\sigma f = \sqd \sumeps \sigma_I  \langle f,h_I^{\varepsilon} \rangle h_I^{\varepsilon}.\]
The condition \( \|\sigma\|_{\infty, W} < \infty\) is equivalent to the boundedness of \(T_{\sigma}\) on \(\ltw\) (see, e.g., Proposition 1.6 in \cite{IsKwPo14} and 
Theorem 5.2 in \cite{BiPeWi14} and for an explicit statement; it is also contained in \cite{TrVo97}). The martingale transform is considered a good model for Calder\'{o}n-Zygmund singular integral operators. 

We define the function \(N:[1, \infty) \to [1, \infty)\) by
\[ N(X)=  \sup \|T_{\sigma}\|_{\ltw \to \ltw}, \]
where the supremum is taken over all \( d \times d\) matrix \(A_2^d\) weights \(W\) with \([W]_{A_2^d} \leq X\) and \(\|\sigma\|_{\infty, W} \leq 1\). It was shown in \cite{BiPeWi14} that
\begin{equation} \label{cubicbound}
N(X) \lesssim (\log X ) X^{3/2}.
\end{equation}
From the estimates (\ref{eq:sq}) and (\ref{eq:lower}) mentioned above, one obtains immediately the slightly improved bond
\begin{equation} \label{eq:mbound}
    \| T_\sigma \|_{\ltw \to \ltw} \le C_{p,d}  \log([W]_{A_2})^{1/2}  [W]_{A_2} [W]_{A_\infty})^{1/2}  
\end{equation}
and thus
\( N(X) \lesssim (\log X)^{1/2} X^{3/2}.\) The linear bound \(N(X) \lesssim X\) is conjectured, but this seems out of reach at the moment.
If the conjectured lower bound for the matrix weighted dyadic square function (\ref{eq:conj}) holds, (\ref{eq:mbound}) could be improved to
$$
 \| T_\sigma \|_{\ltw \to \ltw} \le C_{p,d}    [W]^{1/2}_{A_2} [W]_{A_\infty}^{1/2} [W^{-1}]_{A_\infty}^{1/2}  ,
$$
which is the best currently known bound for scalar kernels and matrix weights, even in the case of scalar Lerner operators \cite{NaPeTrVo17}.

We will now go one step further and allow off-diagonal terms of the operator in the Haar expansion:
\begin{definition}  \label{def:hshift}
A (cancellative) dyadic \(W\)-Haar shift on \(\mathbb{R}^p\) of parameters \((m,n)\), with \(m,n \in \mathbb{N} \cup \{0\}\), is an operator of the form
\[S f = \sum_{L \in \dd}
\sum_{\varepsilon, \varepsilon' \in \{0, 1\}^p \setminus \{0\}} 
\sum_{\substack{
		I \in \mathcal{D}_m(L) \\
		J \in \mathcal{D}_n(L)}}
A^{L}_{I,J, \eps, \eps'} \langle f, h_I^{\varepsilon} \rangle_{\ltrp} h_J^{\varepsilon'},\]
where \(A^{L}_{I,J,\eps, \eps'} \in \mathcal{M}_d(\mathbb{C})\) such that  \( \left \|\langle W \rangle_L ^{1/2} A^{L}_{I,J\eps, \eps'} \langle W \rangle_L ^{-1/2}\right\| \leq \frac{\sqrt{|I|} \sqrt{|J|}}{|L|} =2^{-\frac{(m+n)p}{2}},\) and \(f\) is any locally integrable function. The number \(k := \max\{m,n\}+1\) is called the complexity of the \(W\)-Haar shift.
\end{definition}

For \(0 \leq t \leq k-1\) we introduce the notation \(\mathcal{L} _t :=  \{ I \in \mathcal{D} : \ell(I)=2^{t+k q},  q \in \mathbb{Z} \},\) and define the slice \(S_t\) of the shift \(S\) by 
\[S_t f = \sum_{L \in \mathcal{L}_t}
\sum_{\varepsilon, \varepsilon' \in \{0, 1\}^p \setminus \{0\}} 
\sum_{\substack{
		I \in \mathcal{D}_m(L) \\
		J \in \mathcal{D}_n(L)}}
A^{L}_{I,J\eps, \eps'} \langle f, h_I^{\varepsilon} \rangle_{\ltrp} h_J^{\varepsilon'}, \]
where \(S_t\) also acts on locally integrable functions. We can thus decompose \(S\) as \(S=\sum _{t=0} ^{k-1} S_t\). These slices
  \(S_t\) can be seen as martingale transforms when we are moving \(k\) units of time at once, which we will exploit in the proof of Theorem \ref{mainshift} below in Section \ref{four}.
  
  \begin{theorem}\label{mainshift}
	Let \(W\) be a matrix \(A_2^d\) weight and \(S\) be a dyadic \(W\)-Haar shift on \(\mathbb{R}^p\) of complexity \(k \). Then
	\[\|S\|_{\ltw \to \ltw} \leq C_p \cdot k p d^3 N([W]_{A_2^d}),\]
	where $C_p$ is a constant depending only on $p$.
\end{theorem}

\subsection{Matrix BMO space and paraproducts}
We now introduce the appropriate notion of BMO space for our $T1$ theorem:
\begin{definition}[\protect{\cite{IsKwPo14}}]\label{bmow}
If \(W\) is a matrix \(A_2\) weight, we say that a locally integrable function \(B:\mathbb{R}^p \to \mathcal{M}_d(\mathbb{C})\) belongs to \(BMO_W(\R^p)\) if 
\[\sup_{I \ \rm{cube}} \frac{1}{|I|} \int_I \|W^{1/2}(x) (B(x) - \langle B \rangle_I) \langle W \rangle_I ^{-1/2}\|^2 \mathrm{d}x < \infty.\]
The corresponding dyadic space \(BMO_W^d\) is obtained by taking the supremum only over dyadic cubes \(I \in \dd\).
\end{definition}

It was proved in \cite{IsKwPo14}, Theorem 1.3 and Corollary 1.5,
 that for a matrix $A_2$ weight $W$, this  \(BMO_W^d\) norm is equivalent to the norm given by the square root of
\begin{equation}  \label{eq:wbmod}
\sup_{J \in \dd} \frac{1}{|J|} \sum_{I \subseteq J} \sum_{\varepsilon \in \{0, 1\}^p \setminus \{0\}} \left\|  \langle W \rangle_I^{1/2} \left(B_I^{\varepsilon}\right)^*  B_I^{\varepsilon}  \langle W \rangle_I^{-1/2} \right\|,      
\end{equation}
where \(B_I^{\varepsilon}\) is the matrix of Haar coefficients of the entries of \(B\) with respect to \(h_I^{\varepsilon}\).
 In case that $B$ is a scalar-valued function, this coincides 
with the usual scalar dyadic BMO norm. We will therefore consider the norm given by  (\ref{eq:wbmod}) as the natural norm on  \(BMO_W^d\).

We will now introduce the dyadic paraproducts with matrix coefficients and their adjoints.
 \begin{definition}
If \(B:\mathbb{R}^p \to \mathcal{M}_d(\mathbb{C})\) and \(f:\mathbb{R}^p \to \mathbb{C}^d\) are locally integrable functions, then the dyadic paraproduct \(\Pi_B\) with respect to a dyadic grid \(\dd\) is defined as
\[\Pi_B f = \sd \sumeps B_I^{\varepsilon} \langle f \rangle_I h_I^{\varepsilon}.\] 
The formal adjoint of \(\Pi_B\) is the operator \(\Pi_{B^*}^* = (\Pi_B)^*\) given by
\[\Pi_B^* f = \sd \sumeps (B^*)_I^{\varepsilon} \langle f, h_I^{\varepsilon} \rangle_{\ltrp} \frac{\chi_I}{|I|} .\] 
 \end{definition}
  It was shown in \cite{IsKwPo14} that if \(W\) is a matrix \(A_2\) weight, then \(\Pi_B\) is bounded on \(\ltw\) if and only if \(B \in BMO_W^d.\)
 For $X \ge 1$, we thus write $C_\Pi(X)$ for the smallest constant such that
 $$
       \|\Pi_B \|_{\ltw \to \ltw} \le C_\Pi(X) \| B \|_{BMO_W^d}
 $$
 for all $B \in BMO_W^d$ and all matrix $A_2$ weights $W$ with $[W]_{A_2} \le X$. Here as in the following, \(BMO_W^d\) is equipped with the norm given by (\ref{eq:wbmod}).
It was shown in the proof of Theorem 1.3, (b) $\Rightarrow$ (a) in  \cite{IsKwPo14} that 
\begin{equation}
 \|S_W \Pi_B\|_{\ltw \to \ltrp} \le C_{p,d} \|M'_W\|_{\ltw \to \ltrp}  \| B \|_{BMO_W^d}   \le \tilde C_{p,d} [W]_{A_2}^{1/2}    [W]_{A_\infty}^{1/2}  \| B \|_{BMO_W^d},
\end{equation}
 and this is sharp, by comparison with the scalar case. Hence, using the lower bound in (\ref{eq:lower}), we obtain
 \begin{equation} \label{eq:para}
 \| \Pi_B\|_{\ltw \to \ltw} \le  \tilde C_{p,d}  \log([W]_{A_2})^{1/2} [W]_{A_2}   [W]_{A_\infty}^{1/2}   \| B \|_{BMO_W^d}.
 \end{equation}
and thus
 \begin{equation} \label{eq:parax}
 C_\Pi(X) \lesssim   \log(X)^{1/2} X^{3/2}.
 \end{equation}

We are now ready to consider Calder\'{o}n-Zygmund operators with matrix kernels adapted to a weight $W$. Since the different values of the signature \(\varepsilon\) do not play an important role, we will usually omit the sum in \(\varepsilon\) from now on.

\subsection{Matrix-weighted Calder\'{o}n-Zygmund operators}

\begin{definition}
Let \(\Delta=\{(x,x): x \in \mathbb{R}^p\}\) be the diagonal of \(\mathbb{R}^p \times \mathbb{R}^p\) and let \(W\) be a matrix weight. We say that a function \(K: \mathbb{R}^p \times \mathbb{R}^p \setminus \Delta \to \mathcal{M}_d(\mathbb{C})\) is a standard \(W\)-Calder\'{o}n-Zygmund kernel, if there exists \(\delta>0\) and constants \(C_0, C_{\delta}\) such that
\[\|\langle W \rangle_I ^{1/2} K(x,y) \langle W \rangle_I ^{-1/2}\| \leq \frac{C_0}{|x-y|^p},\]
\[\|\langle W \rangle_I ^{1/2} (K(x,y)-K(x',y)) \langle W \rangle_I ^{-1/2}\|+\|\langle W^{-1} \rangle_I ^{1/2} (K^*(y,x)-K^*(y,x')) \langle W^{-1} \rangle_I ^{1/2} \| \leq C_{\delta} \frac{|x-x'|^{\delta}}{|x-y|^{p+\delta}},\]
for all cubes \(I \subset \mathbb{R}^p\) and all points \(x,x' \in I, y \in \mathbb{R}^p\) with \(|x-y|>2|x-x'|\).
\end{definition}

The notion of $W$-Calder\'on-Zygmund kernels, albeit with more restrictive conditions, was first introduced by J. Isralowitz (see \cite{Is15}).

Given a $W$-Calder\'on-Zygmund kernel $K$,
an operator \(T\), defined on the class of step functions (which is dense in \(L^2(\mathbb{R}^p)\)), is called a \(W\)-Calder\'{o}n-Zygmund operator on \(\mathbb{R}^p\) associated with \(K\), 
if itsatisfies the kernel representation
\[Tf(x)=\int_{\mathbb{R}^p} K(x,y)f(y)\, \mathrm{d}y, \qquad x \notin {\mathrm{supp}}\,f .\]

Note that for a matrix $A_2$ weight $W$, it follows immediately from the definition that $K(x,y)$ is a standard $W$-Calder\'{o}n-Zygmund kernel if and only if
the adjoint kernel $K^*(y,x)$ is a $W^{-1}$-Calder\'{o}n-Zygmund kernel. Moreover, the $W$-Calder\'on-Zygmund operator $T$ is associated to $K(x,y)$, if and only if the 
$W^{-1}$-Calderon-Zygmund operator $T^*$ is associated to $K^*(y,x)$.

Generally, for suitable scalar functions $f$ and $g$, we will write  $\langle Tf, g \rangle$ for the $d \times d$ matrix with entries
$$
(\langle Tf, g\rangle)_{i j} = \langle T f e_j,  g e_i \rangle = \langle Tf, g e_i \otimes e_j \rangle_{\mathcal{S}_2 \otimes \ltrp}.
$$

Furthermore, we want to say that  \(T\) satisfies the \emph{\(W\)-weak boundedness property}, if there exists a constant  $C_{WBP} >0$ such that
\begin{equation}   \label{eq:wbc}
\|\langle W \rangle_{I}^{1/2} T \langle W \rangle_{I}^{-1/2} \chi_{I_i}, \chi_{I_i} \rangle \| \leq C_{WBP} |I| ,
\end{equation}
for all cubes \(I\) and all first-generation dyadic children \(I_i\) of \(I\) . 

Let \(T\) be a \(W\)-Calder\'{o}n-Zygmund operator as above which satisfies the $W$-weak boundedness property (\ref{eq:wbc}). 
Even though \(T\) does not formally act on the constant function \(1\), we can define the Haar coefficients of \(T1\) (and, similarly, of \(T^*1\)) in the following way:
%In order to make sense of the paraproduct \(\Pi_{T1}\), we just have to consider its Haar coefficients \(\langle T1, h_I \rangle\). Here \(T1\) is an \(\mathcal{M}_d(\mathbb{C})\)-valued function of \(\mathbb{R}^p\) and \(\langle T1, h_I \rangle\) is the matrix whose entries are given by 
We have
\begin{align*}
\langle T1, h_I \rangle & = \langle 1, T^* h_I \rangle = \int_{\mathbb{R}^p} (T^* h_I)^*(x) \mathrm{d}x \\
& = \int_{3I} (T^* h_I)^*(x) \mathrm{d}x + \int_{(3I)^c} (T^* h_I)^*(x) \mathrm{d}x \\
& =  \int_{3I} (T^* h_I)^*(x) \mathrm{d}x + \int_{(3I)^c} \int_I K(y, x) h_I(y) \mathrm{d}y \mathrm{d}x \\
& =  \int_{3I} (T^* h_I)^*(x) \mathrm{d}x + \langle W \rangle_I ^{-1/2}  \int_{(3I)^c} \int_I \langle W \rangle_I ^{1/2} (K(y, x) - K(c_I, x)) \langle W \rangle_I ^{-1/2}  h_I(y)  \mathrm{d}y \mathrm{d}x\, \langle W \rangle_I ^{1/2},
\end{align*}
where \(c_I\) is the centre of \(I\). The first integral is well-defined by repeated use of the $W$-weak boundedness property (\ref{eq:wbc}). For the second part, we have
\begin{align*}
& \big \|  \langle W \rangle_I ^{-1/2} \big \| \big \|  \langle W \rangle_I ^{1/2} \big \| \int_{(3I)^c} \int_I \big \|\langle W \rangle_I ^{1/2} (K(y, x) - K(c_I, x)) \langle W \rangle_I ^{-1/2} \big \| |I|^{-1/2}  \mathrm{d}y \mathrm{d}x \\
& \leq   \big \|  \langle W \rangle_I ^{-1/2} \big \| \big \|  \langle W \rangle_I ^{1/2} \big\|  |I|^{-1/2} C_{\delta} \int_{(3I)^c} \int_I  \frac{|y-c_I|^{\delta}}{|x-y|^{p+\delta}}  \mathrm{d}y \mathrm{d}x \\
& \leq C_{\delta}  \big \|  \langle W \rangle_I ^{-1/2} \big \| \big \|  \langle W \rangle_I ^{1/2} \big\|  |I|^{-1/2}   \int_{(3I)^c} \frac{\ell(I)^{\delta}}{dist(x, I)^{p+\delta}} \mathrm{d}x |I| < \infty.
\end{align*}
Here, $\langle T1, h_I \rangle$ is the matrix with the entries
\[(\langle T1, h_I \rangle)_{i j} = \langle T1 e_j,  h_I e_i \rangle = \langle T1, h_I e_i \otimes e_j \rangle_{\mathcal{S}_2 \otimes \ltrp} := \int_{\mathbb{R}^p} tr\big(T1(x) (e_j \otimes h_I(x) e_i)\big) \mathrm{d}x,\]
where \(\langle \cdot, \cdot \rangle_{\mathcal{S}_2}\) denotes the Hilbert-Schmidt inner product on the space of \(d \times d\) matrices, and \(\{e_1, \ldots, e_d\}\) is the standard orthonormal basis of \(\mathbb{C}^d\).

Since the Haar coefficients $\langle T1, h_I \rangle$ and $\langle T^*1, h_I \rangle$ are well-defined, we can thus 
 give meaning to the operators \(\Pi_{T1}\) and \(\Pi^*_{(T^*1)^*}\), respectively.

We can now state the main result of the paper. 

\begin{theorem}\label{mainthm}
Let \(W\) be a \(d \times d\) matrix \(A_2\) weight on \(\mathbb{R}^p\). Let \(T\) be a \(W\)-Calder\'{o}n-Zygmund operator on \(\mathbb{R}^p\) associated to the matrix kernel \(K\). 
Then the following are equivalent:
\begin{enumerate}
\item \(T\) satisfies the \(W\)-weak boundedness property,  \(T1 \in BMO_W\), and \(T^*1 \in BMO_{W^{-1}} \).
\item   \(T:\ltw \to \ltw\) is bounded.
\item   \(T^*:\ltwi \to \ltwi\) is bounded.
\end{enumerate}
In particular, we have in this case the representation
\begin{equation}   \label{eq:wrep}
   \langle  Tf, g \rangle = C \E_\Omega \sum_{n,m} \tau(m,n)  \langle \S_{m,n}^\omega f,g \rangle + \E_\Omega \langle \Pi^\omega_{T1} f,g \rangle+    \E_\Omega \langle \left(\Pi^\omega_{T^*1}\right)^* f,g \rangle
\end{equation}
for all $f \in  C^1_c(\mathbb{R}^p, \C^d) \cap \ltw$ and all $g \in C^1_c(\mathbb{R}^p, \C^d) \cap \ltwi$. Here, the $\S_{m,n}^\omega$ are cancellative dyadic W-Haar shifts with respect to the grid $\dd^\omega$, 
$(\tau(m,n)) \lesssim 2^{-\delta(m+n)/4}(\max(m,n) +1)$, and 
\(C\) is a constant depending only on the constants $C_0$, $C_\delta$ in the \(W\)-Calder\'{o}n-Zygmund kernel conditions, the \(W\)-weak boundedness constant $C_{WBC}$, \(p\) and \(d\).
\end{theorem}

\begin{corollary}   \label{cor:bound}
We have the estimate
\[\|T\|_{\ltw \to \ltw}\leq C  \cdot N([W]_{A_2}) +C_{\Pi}([W]_{A_2}) \left(   \|T1\|_{BMO_W} + \|T^*1\|_{BMO_{W^{-1}}} \right) ,\]
where \(C\) depends only the constants $C_0$, $C_\delta$ in the \(W\)-Calder\'{o}n-Zygmund kernel conditions, the \(W\)-weak boundedness constant $C_{WBC}$, \(p\) and \(d\). In particular,
$$
      \|T\|_{\ltw \to \ltw}\leq C (1+   \|T1\|_{BMO_W} + \|T^*1\|_{BMO_{W^{-1}}})  \log([W]^{1/2}_{A_2}) [W]^{3/2}_{A_2}.
$$
\end{corollary}

\begin{remark} If the conjectured lower bound (\ref{eq:conj}) holds for the weighted square function $S_W$, then this can be improved to
$$
      \|T\|_{\ltw \to \ltw}\leq C (1+  \|T1\|_{BMO_W} + \|T^*1\|_{BMO_{W^{-1}}}) [W]^{3/2}_{A_2},
$$
which is the best currently known bound for matrix weights and scalar kernels, see \cite{NaPeTrVo17}.
\end{remark}

\subsection{Good cubes, bad cubes,  and the Representation Theorem}
In the proof of our main result we will need the notion of "good" cubes, which was introduced in \cite{NaTrVo03}.

\begin{definition}
Let us fix a large parameter \(r \in \mathbb{N}.\) We say that a cube \(I \in \dd_{\omega}\) is \emph{bad}, if there exists \(J \in \dd_{\omega}\) such that \(\ell(J) \geq 2^r \ell(I)\) and 
\[dist(I, \partial J) \leq \ell(I)^{\gamma} \ell{J}^{1-\gamma},\]
where \(\gamma = \frac{\delta}{4(\delta + p)}.\) A cube \(I \in \dd_{\omega}\) is called \emph{good}, if it is not bad. 
\end{definition}
Note that our choice of $\gamma$ differs by a factor $2$ from the usual one.

As was shown in \cite{Hy11}, we can  fix \(r\) large enough such that 
\[\pi_{\rm{bad}} := \mathbb{P}_{\Omega}(\{\omega: I \dotplus \omega \ \rm{is\ bad}\}) < 1.\]
We note that this probability is independent of the cube \(I \in \dd^0\). We also define \(\pi_{\rm{good}} := 1 - \pi_{\rm{bad}}\).
The proof of our main result is based on the following random expansion of an operator \(T\) in terms of Haar functions \(h_I\), where the bad cubes are discarded.

\begin{proposition} [\protect{Hyt\"{o}nen \cite{Hy11}, H\"{a}nninen, Hyt\"{o}nen \cite{HaHy16}}] \label{discard}
%Let \(T\) be a Calder\'{o}n-Zygmund operator on \(\mathbb{R}^p\) that is bounded on  \(\ltrp\). Then, f
For all \(f,g \in C^1_c(\mathbb{R}^p, \C^d),\) we have the expansion
\[\langle T f,g \rangle_{\ltrp} = \frac{1}{\pi_{\mathrm{good}}} \mathbb{E}_{\Omega} 
\sum_{\substack{I, J \in \dd^{\omega}\\
		         \mathrm{smaller}\{I, J\} \mathrm{ is\,  good}}}
\big \langle \langle T h_I, h_J \rangle \langle f, h_I\rangle_{\ltrp}, \langle g, h_J\rangle_{\ltrp} \big \rangle_{\mathbb{C}^d},\]
where 
\[ \mathrm{smaller}\{I, J\} := \begin{cases}
I,& \text{if } \ell(I) \leq \ell(J),\\
J,              & \text{otherwise}.
\end{cases}\]
\end{proposition}

\begin{remark}
This follows as in the proof of \cite{Hy11}, Prop. 3.5 (see also \cite{HaHy16}, Cor. 6.3). These papers use slightly different conditions on $T$, but  we only need here that the inner product 
$\langle T f,g \rangle$ can be expanded in the Haar basis for each of the dyadic grids $\dd^\omega$, and this is for example ensured by \(f,g \in C^1_c(\mathbb{R}^p)\), a (large) a priori bound on the norms of $W$, $W^{-1}$, and the conditions (1) in the $T1$ Theorem \ref{mainthm}, see Lemma \ref{joshlemma} below.
We should mention that this version is a particular case of Corollary 6.3 in \cite{HaHy16}. We only need the result for \(\mathbb{C}^d\)- valued functions instead of functions taking values in an arbitrary Banach space \(E\), whereas our kernels are matrix-valued (they are operator-valued in \cite{HaHy16}). Therefore, the Rademacher \(R\)-bounds reduce to uniform bounds in our case. 
\end{remark}

Let us first mention how Proposition \ref{discard} can be used to prove the main result. For the moment, we fix \(\omega \in \Omega\) and focus on the sum inside \(\mathbb{E}_{\Omega}\); for notational ease, we also drop the index \(\omega\). 

Following \cite{Is15}, we extract the paraproducts by considering the operator \(\widetilde{T} := T - \Pi_{T1}- \Pi^*_{(T^*1)^*}.\)

We will now show how to identify the sum involving \(\widetilde{T}\) as a sum of dyadic \(W\)-shifts. In order to do that, the sum is rearranged according to the minimal common dyadic ancestor of \(I\) and \(J\), which is denoted by \(I \vee J\) (if \(I \subseteq J\), then \(I \vee J = J\); if \(I \cap J = \varnothing\), Lemma 3.7 in \cite{Hy11} shows the existence of a common dyadic ancestor).

Splitting the sum according to which of the cubes \(I\) and \(J\) has smaller side length (and hence is good), then rearranging the sum according to which cube \(L\) is the minimal common dyadic ancestor \(I \vee J\), and what the size of \(I\) and \(J\) relative to \(L\) is, we obtain 
\[\sum_{\substack{I, J:\\
		\mathrm{smaller}\{I, J\} \mathrm{ is\,  good}}} = 
\sum_{m \geq n} \sum_{L} 
\sum_{\substack{I, J: I \vee J = L\\
		        I \mathrm{\, is\,  good}\\
		        \ell(I) = 2^{-m} \ell(L)\\
		        \ell(J) = 2^{-n} \ell(L)}}
+ \sum_{m < n} \sum_{L} 
\sum_{\substack{I, J: I \vee J = L\\
		J \mathrm{\, is\,  good}\\
	\ell(I) = 2^{-m} \ell(L)\\
	\ell(J) = 2^{-n} \ell(L)}}.\]

If we write 
\[\sum_{L} 
\sum_{\substack{I, J: I \vee J = L\\
		I \mathrm{\, is\,  good}\\
		\ell(I) = 2^{-m} \ell(L)\\
		\ell(J) = 2^{-n} \ell(L)}}
\big \langle \langle \tilde{T} h_I, h_J \rangle \langle f, h_I\rangle_{\ltrp}, \langle g, h_J\rangle_{\ltrp} \big \rangle_{\mathbb{C}^d} =: \langle S_{m n} f, g \rangle_{\ltrp},\]
we get that
\begin{equation}\label{expn}
\langle T f,g \rangle_{\ltrp} = \frac{1}{\pi_{\mathrm{good}}} \mathbb{E}_{\Omega}  \sum_{m, n} \big \langle S^{\omega}_{m n} f, g  \big \rangle_{\ltrp} + \frac{1}{\pi_{\mathrm{good}}} \mathbb{E}_{\Omega} \big \langle (\Pi^{\omega}_{T1} + (\Pi^{\omega}_{T^*1})^*) f, g \big \rangle_{\ltrp}.
\end{equation}

\section{The proof of Theorem \ref{mainthm}}   \label{three}
The equivalence of $(2)$ and $(3)$ in Theorem \ref{mainshift} is immediate. We will first consider the sufficiency direction $(1) \Rightarrow (2)$. 

Since we already have the necessary bounds for the paraproducts by (\ref{eq:para}) and (\ref{eq:parax}), it remains to study the dyadic shifts 
$S_{m,n}$ that appear in the expansion (\ref{expn}) of the operator. 
We will show in this section that up to a constant $C$ depending only on $C_0$, $C_\delta$ in the \(W\)-Calder\'{o}n-Zygmund kernel conditions, the \(W\)-weak boundedness constant $C_{WBC}$, \(p\) and \(d\)
the dyadic shifts \(S_{m n}\) are of the form $2^{-\delta(m+n)/4} \S_{m,n}$, where $\S_{m,n}$ is a dyadic  \(W\)-Haar shifts in the sense of Definition \ref{def:hshift}. Theorem \ref{mainshift} then yields the necessary bounds for these shifts, which guarantee convergence,
thereby proving the representation formula (\ref{eq:wrep}), and moreover prove Corollary \ref{cor:bound}.
 
Before we proceed, let us introduce some more useful notations. We first fix a dyadic lattice \(\dd\) in \(\mathbb{R}^p\); all dyadic operators will be considered with respect to this grid \(\dd\). As before, let \(\widetilde{T} := T - \Pi_{T1}- \Pi^*_{(T^*1)^*}\). For $I, J \in \dd$, define the matrix \(\widetilde{T}_{I,J}\) as 
\(\widetilde{T}_{I,J} = \langle \widetilde{T}h_I, h_J \rangle\) (also define \(T_{I, J}\) in a similar way). Moreover, for any fixed dyadic cube \(L\), let \(\widetilde{T}^L := \langle W \rangle_L^{1/2} \widetilde{T} \langle W \rangle_L^{-1/2}\) and \(\widetilde{T}^L_{I, J} := \langle W \rangle_L^{1/2} \widetilde{T}_{I, J} \langle W \rangle_L^{-1/2}\) (and similarly define \(T^L\) and \(T^L_{I, J}\)). 

In the following lemma we prove that the dyadic shifts \(S_{m n}\) are \(W\)-Haar shifts. We are only considering the case \(m \geq n\) (which means \(\ell(I) \leq \ell(J)\)), since the case \(m < n\) can be treated similarly by duality. A version of this lemma was stated and proved in \cite{Is15}, and the proof here runs along the same lines. However, we work with different notions of $W$-weak boundedness and $W$-Calderon-Zygmund kernels, which changes some arguments. 
For clarity, we give the whole proof here.

\begin{lemma}\label{joshlemma}
Let \(m, n \ge 0\), \(n \leq m\).

Let \(W\) be a matrix \(A_2\) weight on \(\mathbb{R}^p\) and let \(T\) be a \(W\)-Calder\'{o}n-Zygmund operator on \(\mathbb{R}^p\), which satisfies the $W$-weak boundedness condition  in
Theorem \ref{mainthm}.

Fix a cube \(L \in \dd\). If \(I\) and \(J\) are two cubes such that \(L\) is their smallest common dyadic ancestor, \(\ell(I) = 2^{-m} \ell(L), \ell(J) = 2^{-n} \ell(L),\) and \(I\) is good,  then 
\[
\left\|\widetilde{T}^{L}_{I,J}\right\| \leq C \frac{\ell(I)^{\frac{p + \delta}{2}} \ell(J)^{\frac{p + \delta}{2}}}{D(I,J)^{p + \delta}}  \le \tilde C  \frac{\sqrt{|I|} \sqrt{|J|}}{|L|} 2^{-\frac{(m+n)\delta}{4}},
\]
where the long distance \(D(I, J)\) is defined as \(D(I, J) := \ell(I) + \ell(J) + dist(I, J)\),  and $C$, $\tilde C$
 are constants depending only on $p$,$d$, the smoothness and decay constants $C_\delta$, $C_0$
 of the $W$-Calderon-Zygmund
kernel, and the $W$-weak boundedness constant $C_{WBC}$.

\end{lemma}

\begin{proof}
Before we proceed, note that \(\Pi_{T1}h_I \) is contained in the span of the Haar functions $h_K$, $K \subsetneq I$, with  \(d \times d\) matrix coefficients, 
 and \(\Pi^*_{(T^*1)^*} h_I \) is a multiple of $\frac{\chi_I}{|I|}$ with a matrix coefficients. In particular, the support of both \(\Pi_{T1}h_I \) and \(\Pi^*_{(T^*1)^*} h_I \) is contained in $I$.

Following \cite{Hy11} and \cite{Is15}, we decompose the set \(\Gamma := \{(I, J) \in \dd \times \dd: \ell(I) \leq \ell(J)\}\) as
\begin{align*}
\{(I, J) \in \dd \times \dd: \ell(I) \leq \ell(J)\} & = \{(I, J) \in \dd \times \dd: I \subsetneq J\} \cup \{(I, J) \in \dd \times \dd: I = J\}\\
&  \cup \{(I, J) \in \dd \times \dd: dist(I, J) > \ell(I)^{\gamma} \ell(J)^{1 - \gamma}\} \\
& \cup \{(I, J) \in \dd \times \dd: I \cap J = \varnothing,\ dist(I, J) \leq \ell(I)^{\gamma} \ell(J)^{1 - \gamma}\} \\
& =: \Gamma_{\mathrm{in}} \cup \Gamma_{\mathrm{equal}} \cup \Gamma_{\mathrm{out}} \cup \Gamma_{\mathrm{near}}.
\end{align*}

We will now estimate \(\left\|\tilde{T}^{L}_{I,J}\right\|\) for \((I, J)\) in each of these sets. 

\noindent
Case 1: \((I, J) \in \Gamma_{\mathrm{in}}\)

It is not difficult to check that \((\tilde{T}^L)^* 1 = 0,\) in the sense that \(\int_{\mathbb{R}^p} \tilde{T}^L h_I(x) \mathrm{d}x = 0\) for each cube \(I\). Indeed, we have
\begin{align*}
\langle h_I, (\tilde{T}^L)^* 1  \rangle & = \langle \tilde{T}^L h_I, 1  \rangle= \left \langle \langle W \rangle_L ^{1/2}  (T - \Pi_{T1} - \Pi^*_{(T^*1)^*})  \langle W \rangle_L ^{-1/2} h_I, 1 \right \rangle\\
& = \langle T^L h_I, 1  \rangle - \left \langle \langle W \rangle_L ^{1/2}  
  \Pi_{T1} \langle W \rangle_L ^{-1/2} h_I, 1 \right \rangle - \left \langle \langle W \rangle_L ^{1/2} \Pi^*_{(T^*1)^*}  \langle W \rangle_L ^{-1/2} h_I, 1 \right \rangle.
\end{align*}

Using the definitions of the paraproduct and its adjoint, we get that \(\left \langle \langle W \rangle_L ^{1/2}  \Pi_{T1} \langle W \rangle_L ^{-1/2} h_I, 1 \right \rangle = 0,\) since \(\langle h_K, 1\rangle_{\ltrp} = 0\) for any \(K \in \dd\), and
\begin{align*}
\left \langle \langle W \rangle_L ^{1/2}  \Pi^*_{(T^*1)^*} \langle W \rangle_L ^{-1/2} h_I, 1 \right \rangle & = \left \langle \langle W \rangle_L ^{1/2}  \langle (T^*1)^*, h_I \rangle  \langle W \rangle_L ^{-1/2} \frac{\chi_I}{|I|} , 1 \right \rangle\\
& = \langle W \rangle_L ^{1/2}  \langle h_I, T^*1 \rangle  \langle W \rangle_L ^{-1/2} \\
& = \langle W \rangle_L ^{1/2}  \langle T h_I, 1 \rangle  \langle W \rangle_L ^{-1/2}.
\end{align*}

The last term is equal to 
\[\left \langle \langle W \rangle_L ^{1/2}  T h_I \langle W \rangle_L ^{-1/2}, 1 \right \rangle = \left \langle \langle W \rangle_L ^{1/2}  T \langle W \rangle_L ^{-1/2} h_I, 1 \right \rangle,\]
which is exactly \(\langle T^L h_I, 1 \rangle\). Therefore \(\int_{\mathbb{R}^p} \tilde{T}^L h_I(x) \mathrm{d}x = \langle \tilde{T}^L h_I, 1  \rangle =  0\).

Let \(J_I\) be the child of \(J\) that contains \(I\). Then 
\begin{align*}
\langle \tilde{T}^L h_I , h_J \rangle & = \langle \tilde{T}^L h_I , \chi_{J_I^c} h_J \rangle+ \langle \tilde{T}^L h_I , \chi_{J_I} h_J \rangle\\
& = \langle \tilde{T}^L h_I , \chi_{J_I^c} h_J \rangle + \langle h_J \rangle_{J_I} \langle \tilde{T}^L h_I , \chi_{J_I} \rangle \\
& = \langle \tilde{T}^L h_I , \chi_{J_I^c} (h_J - \langle h_J \rangle_{J_I}) \rangle +  \langle h_J \rangle_{J_I} \langle \tilde{T}^L h_I , 1 \rangle,
\end{align*} 
hence \(\left\|\tilde{T}^{L}_{I,J}\right\| =  \left\|\langle \tilde{T}^L h_I , h_J \rangle \right \| \leq 2 |J|^{-1/2} \int_{J_I^c} \|\tilde{T}^L h_I(x) \| \mathrm{d}x\).

To estimate the last integral, we first notice that \(\tilde{T}^L h_I(x) = T^L h_I(x)\) when \(x \in J_I^c\) (this follows from the remark at the beginning of the proof, since \( \langle W \rangle_L ^{1/2}  (\Pi_{T1} + \Pi^*_{(T^*1)^*}) \langle W \rangle_L ^{-1/2} h_I\) is supported on \(I\)).

 If \(\ell(I) \leq \ell(J) \leq 2^r \ell(I),\) then 
\begin{align*}
\left\|\tilde{T}^{L}_{I,J}\right\| &  \leq 2 |J|^{-1/2} \int_{J_I^c} \|\tilde{T}^L h_I(x) \| \mathrm{d}x = 2 |J|^{-1/2} \int_{J_I^c} \|T^L h_I(x) \| \mathrm{d}x \\
& \leq 2  |J|^{-1/2} \left( \int_{3I \setminus I} \left \| \int_I \langle W \rangle_L ^{1/2} K(x, y) \langle W \rangle_L ^{-1/2}h_I(y) \mathrm{d}y  \right\| \mathrm{d}x \right.\\
&  \qquad + \left. \int_{(3I)^c} \left \| \int_I \langle W \rangle_L ^{1/2} (K(x, y) - K(x, c_I)) \langle W \rangle_L ^{-1/2}h_I(y) \mathrm{d}y  \right\| \mathrm{d}x \right) \\
& \lesssim |J|^{-1/2}  \left(C_0 \int_{3I \setminus I} \int_I \frac{1}{|x-y|^p} \mathrm{d}y \mathrm{d}x |I|^{-1/2} + C_{\delta}  \int_{(3I)^c} \frac{\ell(I)^{\delta}}{dist(x, I)^{p + \delta}}  \mathrm{d}x \int_I |h_I(y)|  \mathrm{d}y \right) \\
& \lesssim  |J|^{-1/2}  \left(C_0 |I|^{1/2} +  C_{\delta}  \int_{\ell(I)}^{\infty} \frac{\ell(I)^{\delta}}{r^{p + \delta}}r^{p-1}  \mathrm{d}r |I|^{1/2} \right) \lesssim \left( \frac{|I|}{|J|} \right)^{1/2}  \approx \frac{\ell(I)^{\frac{p + \delta}{2}}\ell(J)^{\frac{p + \delta}{2}}}{D(I, J)^{p + \delta}}.
\end{align*}

On the other hand, if \(\ell(J) > 2^r \ell(I)\) (which is the same as \(\ell(J_I) \geq 2^r \ell(I)\)), we have \(dist(I, J_I^c) > \ell(I)^{\gamma} \ell(J_I)^{1-\gamma} \gtrsim \ell(I)^{\gamma} \ell(J)^{1-\gamma},\) since \(I\) is good. It follows that

\begin{align*}
\left\|\tilde{T}^{L}_{I,J}\right\| & \leq 2 |J|^{-1/2} \int_{J_I^c} \|\tilde{T}^L h_I(x) \| \mathrm{d}x = 2 |J|^{-1/2} \int_{J_I^c} \|T^L h_I(x) \| \mathrm{d}x \\
& \leq 2  |J|^{-1/2}  \int_{J_I^c} \left \| \int_I \langle W \rangle_L ^{1/2} (K(x, y) - K(x, c_I)) \langle W \rangle_L ^{-1/2} h_I(y) \mathrm{d}y  \right\| \mathrm{d}x  \\
& \leq 2 |J|^{-1/2}   C_{\delta}  \int_{J_I^c} \int_I \frac{|y-c_I|^{\delta}}{|x-y|^{p + \delta}} |h_I(y)| \mathrm{d}y  \mathrm{d}x \\
& \leq 2  C_{\delta} |J|^{-1/2} |I|^{1/2} \int_{J_I^c} \frac{\ell(I)^{\delta}}{d(x, I)^{p + \delta}} \mathrm{d}x \lesssim \left( \frac{|I|}{|J|} \right)^{1/2} \ell(I)^{\delta} \int_{\ell(I)^{\gamma} \ell(J)^{1-\gamma}} \frac{1}{r^{p+\delta}} r^{p-1} \mathrm{d}r \\
& \lesssim \left( \frac{|I|}{|J|} \right)^{1/2} \frac{\ell(I)^{\delta}}{[\ell(I)^{\gamma} \ell(J)^{1-\gamma}]^{\delta}} \leq \left( \frac{|I|}{|J|} \right)^{1/2} \frac{\ell(I)^{\delta}}{[\ell(I)^{1/2} \ell(J)^{1/2}]^{\delta}} = \frac{\ell(I)^{\frac{p + \delta}{2}}}{\ell(J)^{\frac{p + \delta}{2}}} \approx \frac{\ell(I)^{\frac{p + \delta}{2}}\ell(J)^{\frac{p + \delta}{2}}}{D(I, J)^{p + \delta}},
\end{align*}
where the last inequality is true since \(\gamma \leq \frac{1}{2}\).

\noindent
Case 2: \((I, J) \in \Gamma_{\mathrm{equal}}\)

By the remark at the beginning of the proof, we have \(\tilde{T}_{I,I} = T_{I, I}\), and thus also \(\tilde{T}^L_{I,I} = T^L_{I, I}\). Note that by minimality of $L$, $L=I$.

To emphasize that the two Haar functions appearing in the definition of the matrix \(T^L_{I, I}\) are not the same (even though \(I = J\) in this case), we briefly reintroduce the superscripts \(\varepsilon, \varepsilon'\). 

If \(\{I_i\}_{i=1}^{2^p}\) are the dyadic children of \(I\), then using the kernel representation and the weak boundedness property we obtain
\begin{align*}
\left\|\tilde{T}^{L}_{I,I}\right\| & = \left\|T^{L}_{I,I}\right\| = \left\|\langle T^L h_I^{\varepsilon} , h_I^{\varepsilon'} \rangle\right\| \leq \sum_{i,j=1}^{2^p} \left\|\langle h_I{^\varepsilon} \rangle_{I_i} \langle h_I^{\varepsilon'} \rangle_{I_j} \langle T^L \chi_{I_i} , \chi_{I_j} \rangle \right\| \\
& \leq \sum_{i \neq j} |I|^{-1} \left\|\int_{I_j} \int_{I_i} \langle W \rangle_L ^{1/2} K(x, y) \langle W \rangle_L ^{-1/2} \chi_{I_i}(y) \chi_{I_j}(x) \mathrm{d}y \mathrm{d}x \right\| + \sum_{i=1}^{2^p}  |I|^{-1}  \left\| \langle T^L \chi_{I_i} , \chi_{I_i} \rangle\right\| \\
& \leq C_0 |I|^{-1} \int_{I_j} \int_{I_i} \frac{1}{|x-y|^p} \mathrm{d}y \mathrm{d}x  + \sum_{i=1}^{2^p}  |I|^{-1}  \left\| \langle T^L \chi_{I_i} , \chi_{I_i} \rangle_{\ltrp, \ltrp} \right\| \\
& \lesssim C_0 +  \sum_{i=1}^{2^p}  |I|^{-1} C_{WBP} |I_i| = C_0 + 2^p C_{WBP}.
\end{align*}

\noindent
Case 3: \((I, J) \in \Gamma_{\mathrm{out}}\)

As before, the remark at the beginning of the proof shows that \(\tilde{T}^{L}_{I,J} = T^{L}_{I,J}\) if \(I\) and \(J\) are disjoint (which is obviously the case here).

If \(c_I\) is the centre of \(I\), the decay property of \(K\) and the cancellation of \(h_I\) allow us to estimate

\begin{align*}
\left\|\tilde{T}^{L}_{I,J}\right\| & = \left\|T^{L}_{I,J}\right\|  = \left\| \langle T^L h_I, h_J \rangle \right\| = \left\| \int_J \int_I \langle W \rangle_L ^{1/2} (K(x, y) - K(x, c_I)) \langle W \rangle_L ^{-1/2} h_I(y) h_J(x) \mathrm{d}y \mathrm{d}x \right\|\\
& \leq \int_J \int_I \left \| \langle W \rangle_L ^{1/2} (K(x, y) - K(x, c_I)) \langle W \rangle_L ^{-1/2} \right \| |h_I(y)| |h_J(x)| \mathrm{d}y \mathrm{d}x \\
& \leq |I|^{-1/2} |J|^{-1/2} \int_J \int_I  \frac{|y-c_I|^{\delta}}{|x-y|^{p + \delta}}  \mathrm{d}y \mathrm{d}x \\
& \leq  |I|^{-1/2} |J|^{-1/2} \left( \frac{\ell(I)}{2} \right)^{\delta} \frac{1}{dist(I, J)^{p + \delta}} |I| |J| \leq \frac{\ell(I)^{\delta}}{dist(I, J)^{p + \delta}} \ell(I)^{\frac{d}{2}} \ell(J)^{\frac{d}{2}}. 
\end{align*}

If \(dist(I, J) > \ell(J),\) then \(D(I, J) = \ell(I) + \ell(J) + dist(I, J) < 3 dist(I, J) ,\) thus 
\[\frac{\ell(I)^{\delta}}{dist(I, J)^{p + \delta}} \ell(I)^{\frac{d}{2}} \ell(J)^{\frac{d}{2}} \lesssim \frac{\ell(I)^{
\frac{p + \delta}{2}}\ell(J)^{\frac{p + \delta}{2}}}{D(I, J)^{p + \delta}},\]
since \(\ell(I) \leq \ell(J)\).

On the other hand, if \(dist(I, J) \leq \ell(J),\) then \(D(I, J) = \ell(I) + \ell(J) + dist(I, J) < 3 \ell(J) ,\) so
\[\frac{\ell(I)^{\delta}}{dist(I, J)^{p + \delta}} \ell(I)^{\frac{d}{2}} \ell(J)^{\frac{d}{2}} \leq \frac{\ell(I)^{\delta}}{(\ell(I)^{\gamma}\ell(J)^{1 - \gamma})^{p + \delta}} \ell(I)^{\frac{d}{2}} \ell(J)^{\frac{d}{2}} \le  \frac{\ell(I)^{\frac{p + \delta}{2}} \ell(J)^{\frac{p + \delta}{2}}} {\ell(J)^{p + \delta}} \lesssim \frac{\ell(I)^{\frac{p + \delta}{2}} \ell(J)^{\frac{p + \delta}{2}}} {D(I, J)^{p + \delta}},\]
where in the above inequality we have used that \(\gamma(p + \delta) = \frac{\delta}{4}\).

Case 4: \((I, J) \in \Gamma_{\mathrm{near}}\)

As in the previous case, the disjointness of \(I\) and \(J\) implies that \(\tilde{T}^{L}_{I, J} = T^{L}_{I, J}\). 
Since \(dist(I, J) \leq \ell(I)^{\gamma} \ell(J)^{1-\gamma} \leq \ell(J),\) it follows that \(I \subseteq 5J \setminus J.\) Using the kernel representation of the operator \(T\), we have

\begin{align*}
\left\|\tilde{T}^{L}_{I,J}\right\| & = \left\|T^{L}_{I,J}\right\|  = \left\| \langle T^L h_I, h_J \rangle\right\| = \left\| \int_J \int_I \langle W \rangle_L ^{1/2} K(x, y) \langle W \rangle_L ^{-1/2} h_I(y) h_J(x) \mathrm{d}y \mathrm{d}x \right\|\\
& \leq C_0 |I|^{-1/2} |J|^{-1/2} \int_J \int_I \frac{1}{|x-y|^p} \mathrm{d}y \mathrm{d}x  \leq C_0 |I|^{-1/2} |J|^{-1/2} \int_J \int_{5J \setminus J} \frac{1}{|x-y|^p} \mathrm{d}y \mathrm{d}x \\
& \lesssim C_0 |J|^{-1/2} |J|^{-1/2} |J| = C_0,
\end{align*}
where the last inequality is true, since the goodness of $I$ ensures that  \(\ell(J) \le 2^r \ell(I)$. This completes the proof of the first inequality in the statement of the lemma.

To prove the second inequality, we notice that if \(n \geq 1,\) the minimality of \(L\) implies that \(I\) and \(J\) are disjoint.  Let \(L_I\) be the child of \(L\) that contains \(I\). Since \(L\) is the smallest common dyadic ancestor of \(I\) and \(J\), we have \(dist(I, J) \geq dist(I, \partial L_I)\).
Again by minimality and goodness of $I$,
$$
D(I, J) = \ell(I) + \ell(J) + dist(I, J) \gtrsim \ell(I)^\gamma \ell(L)^{1-\gamma} 
$$
and thus
$$
   \frac{\ell(I)^{\frac{p + \delta}{2}} \ell(J)^{\frac{p + \delta}{2}}}{D(I,J)^{p + \delta}} \lesssim   
                     \frac{\ell(I)^{\frac{p + \delta}{2}} \ell(J)^{\frac{p + \delta}{2}}}{\ell(I)^{\gamma(p + \delta)} \ell(L)^{(1-\gamma)(p + \delta) }}
                     \le   \frac{\sqrt{|I|} \sqrt{|J|}}{|L|}    \frac{\ell(I)^{\frac{\delta}{4}} \ell(J)^{\frac{\delta}{2}}}{\ell(L)^{3 \delta/4 }}
                        \le \frac{\sqrt{|I|} \sqrt{|J|}}{|L|} 2^{-(m+n) \delta/4},
$$
 where we use $\gamma(p+\delta) = \frac{\delta}{4}$.
 
If \(n = 0,\) that is \(J = L,\) we have \(D(I, J) \geq \ell(L)\). Then 
\[ \left\|\tilde{T}^{L}_{I,J}\right\| \leq C \frac{\ell(I)^{\frac{p+\delta}{2}} \ell(J)^{\frac{p+\delta}{2}}} {  \ell(L)^{p+\delta}} = C 2^{-m \frac{p+\delta}{2}} \leq C \frac{\sqrt{|I|} \sqrt{|J|}}{|L|} 2^{-\frac{(m+n)\delta}{2}},\]
which concludes the proof of the lemma.

\end{proof}

This proves that the operators \(S_{m n}^{\omega}\) appearing in \eqref{expn} are appropriate scalar multiples of \(W\)-Haar shifts. Together with Theorem \ref{mainshift} and the estimates
(\ref{eq:para}) and (\ref{eq:parax}), this gives the sufficiency direction $(1)$ $\Rightarrow$ $(2)$ and the claimed bounds. We now prove necessity, namely  $(2)$ $\Rightarrow$ $(1)$.

First, we show that an \(\ltw\)-bounded \(W\)-Calder\'{o}n-Zygmund operator satisfies the \(W\)-weak boundedness property. If \(I\) is a cube and \(I_i\) is a child of \(I\), then
\begin{align*}
&\big \|\big \langle \langle W \rangle_I^{1/2} T \langle W \rangle_I^{-1/2} \chi_{I_i}, \chi_{I_i} \big \rangle \big \| \\
= &\big \|\big \langle \langle W^{1/2} T W^{-1/2} W^{1/2} \langle W \rangle_I^{-1/2} \chi_{I_i}, W^{-1/2} \langle W \rangle_I^{1/2} \chi_{I_i} \big \rangle \big \| \\
 \leq &\|W^{1/2} T W^{-1/2}\|_{\ltrp \to \ltrp} \| W^{1/2} \langle W \rangle_I^{-1/2} \chi_{I_i}\|_{\ltrp} \| W^{-1/2} \langle W \rangle_I^{1/2} \chi_{I_i}\|_{\ltrp}  \\
 \leq &\|T\|_{\ltw \to \ltw} C_p |I_i|^{1/2} [W]_{A_2}^{1/2} |I_i|^{1/2} \leq C_p [W]_{A_2}^{1/2} |I|. 
\end{align*}

To show that \(T1 \in BMO_W\) if \(T :L^2(W) \to L^2(W)\) is bounded, we first introduce the operator 
\[\widetilde{P}_I f = \langle f, h_I \rangle h_I + \sum_{J \subsetneq I} \langle h_I \rangle_{I_J} \langle f, h_J \rangle h_J,\]
where \(I_J\) is the child of \(I\) containing \(J\). Here \(f\) is a locally integrable \(\mathbb{C}^d\)-valued function, but we will also allow the operator \(\widetilde{P}_I \) to act on locally integrable \(\mathcal{M}_d(\mathbb{C})\)-valued functions. Note that since $W$ is $A_2$, $W^{1/2} \widetilde{P}_I W^{-1/2}$ is bounded on $\ltrp$, with bound independent of $I$. We thus have
\begin{align*}
\|W^{1/2} \widetilde{P}_I T \langle W\rangle^{-1/2}_I h_I\|_{\ltrp} & = \| W^{1/2} \widetilde{P}_I W^{-1/2}W^{1/2}T W^{-1/2} W^{1/2}  \langle W \rangle_I^{-1/2} h_I\|_{\ltrp} \\
& \leq \|W^{1/2}  \widetilde{P}_I T W^{-1/2} \|_{\ltrp \to \ltrp} \|W^{1/2}  \langle W \rangle_I^{-1/2} h_I\|_{\ltrp} \\
& \leq C([W]_{A_2}, p, d) \, \|T\|_{\ltw, \ltw}.
\end{align*}

On the other hand, we can write 
\[W^{1/2} \widetilde{P}_I T \langle W\rangle^{-1/2}_I h_I = W^{1/2} \big \langle T \langle W \rangle_I^{-1/2} h_I, h_I \big \rangle h_I + W^{1/2} \sum_{J \subsetneq I} \langle h_I \rangle_{I_J}  \big \langle T \langle W \rangle_I^{-1/2} h_I, h_J \big \rangle h_J.\]
Using the same splitting as in the proof of the first case (\(\Gamma_{\mathrm{in}}\)) of Lemma \ref{joshlemma}, the terms of the previous sum can be expressed as
\[W^{1/2} \langle W \rangle_I^{-1/2}  \langle h_I \rangle_{I_J}  \big \langle T^I \big( \chi_{I_J^c}(h_I - \langle h_I \rangle_{I_J} ) \big), h_J \big \rangle h_J  + W^{1/2} \langle W \rangle_I^{-1/2} (\langle h_I \rangle_{I_J} )^2 \big \langle T^I 1, h_J \big \rangle h_J.\]
It then follows that
\begin{align*}
W^{1/2} \widetilde{P}_I T \langle W\rangle^{-1/2}_I h_I & = W^{1/2} \langle W \rangle_I^{-1/2} \big \langle T^I h_I, h_I \big \rangle h_I + \sum_{J \subsetneq I} W^{1/2} \langle W \rangle_I^{-1/2}  \langle h_I \rangle_{I_J}  \big \langle T^I \big( \chi_{I_J^c}(h_I - \langle h_I \rangle_{I_J} ) \big), h_J \big \rangle h_J \\
& \qquad + \sum_{J \subsetneq I} W^{1/2} \langle W \rangle_I^{-1/2} \big \langle T^I 1, h_J \big \rangle h_J.
\end{align*}
Since the Haar functions form an unconditional basis in \(\ltw\), we have
\begin{multline}
\left\|W^{1/2} \langle W \rangle_I^{-1/2} \big \langle T^I h_I, h_I \big \rangle h_I + \sum_{J \subsetneq I} W^{1/2} \langle W \rangle_I^{-1/2}  \langle h_I \rangle_{I_J}  \big \langle T^I \big( \chi_{I_J^c}(h_I - \langle h_I \rangle_{I_J} ) \big), h_J \big \rangle h_J\right\|^2_{\ltrp} \\
 \leq C([W]_{A_2}, p, d)  \bigg( \big \| \big \langle T^I h_I, h_I \big \rangle h_I \big\|^2_{\ltrp} +  
 \frac{1}{|I|} \sum_{J \subsetneq I}  \big \| \big \langle T^I \big( \chi_{I_J^c}(h_I - \langle h_I \rangle_{I_J} ) \big), h_J \big \rangle h_J \big \|^2_{\ltrp} \bigg)\\
 \lesssim  C([W]_{A_2}, C_\delta, C_0, C_{WBC},  p, d)  \sum_{J \subseteq I} \frac{\ell(I)^{p + \delta}\ell(J)^{p + \delta}}{D(I, J)^{2(p + \delta)}}\\
\leq C([W]_{A_2}, C_\delta, C_0, C_{WBC},  p, d)   \, \sum_{n=0}^{\infty} 2^{-n(p+\delta)} 2^{np} < \infty.
\end{multline}
Here we have used the estimates in the \(\Gamma_{\mathrm{in}}\) and \(\Gamma_{\mathrm{equal}}\) cases from Lemma \ref{joshlemma}.

Altogether, we obtain that \(\frac{1}{|I|^{1/2}} \sum_{J \subsetneq I} W^{1/2} \langle W \rangle_I^{-1/2} \big \langle T^I 1, h_J \big \rangle h_J\) is bounded in \(\ltrp\) with a bound independent of $I$, 
hence also the \(\frac{1}{|I|^{1/2}} \sum_{J \subseteq I} W^{1/2} \langle W \rangle_I^{-1/2} \big \langle T^I 1, h_J \big \rangle h_J\) are uniformly bounded in \(\ltrp\). But this last sum is equal to 
\begin{multline*}
\frac{1}{|I|^{1/2}} \sum_{J \subseteq I}  W^{1/2}  \big \langle T \langle W \rangle_I^{-1/2}  1, h_J \big \rangle h_J = 
\frac{1}{|I|^{1/2}} W^{1/2}  \sum_{J \subseteq I}   \big \langle T 1, h_J \big \rangle h_J \langle W \rangle_I^{-1/2} \\
 = \frac{1}{|I|^{1/2}} W^{1/2} \big( (T1 - \langle T1 \rangle_I) \big) \chi_I \langle W \rangle_I^{-1/2},
 \end{multline*}
and the uniform \(\ltrp\)-boundedness in $I$ of these functions is exactly the condition from Definition \ref{bmow}. Therefore \(T1 \in BMO_W\). 

\(T^*1 \in BMO_{W^{-1}}\) follows immediately by the same argument, since \(T^*:\ltwi \to \ltwi\) is bounded by (3). This finishes the proof of Theorem \ref{mainthm}, up to proving
Theorem \ref{mainshift}.
\qed

\section{The proof of Theorem  \ref{mainshift}}    \label{four}
%Recall Theorem \ref{mainshift}:
%\begin{theorem}
%	Let \(W\) be a matrix \(A_2^d\) weight and \(S\) be a dyadic \(W\)-Haar shift on \(\mathbb{R}^p\) of complexity $k$. Then
%	\[\|S\|_{\ltw \to \ltw} \leq c \cdot k p d^3 N([W]_{A_2^d}),\]
%	where c is an absolute, positive constant.
%\end{theorem}
We now to the proof of the sharp bound for dyadic $W$-Haar shifts, Theorem \ref{mainshift}.
Following the approach in \cite{Tr11}, one can show that it is enough to consider only dyadic \(W\)-Haar shifts on a dyadic system in \(\mathbb{R}\). This reduction is obtained by arranging the dyadic cubes 
in an appropriate way on the real line (for more details in the matrix-weighted case, see also \cite{PoSt15}).

Let \(W\) be a \(d \times d\) matrix \(A_2^d\) weight on \(\mathbb{R}\). For each \(I \in \dd\), choose an orthonormal basis of eigenvectors  \(B_I = \{e_I^1, e_I^2, \ldots, e_I^d\}\) of \( \langle W \rangle_I\) with associated eigenvalues \( \{\lambda_{I, 1}, \lambda_{I, 2}, \ldots, \lambda_{I, d}\}\) , and let \(P_I^i, 1 \leq i \leq d\), be the corresponding orthogonal projection onto the span of \(e_I^i\). 

Since \(S\) is a \(W\)-Haar shift operator of complexity \(k\), it has the form
\[S f = \sum_{L \in \dd}
\sum_{\substack{
            I \in \mathcal{D}_m(L) \\
            J \in \mathcal{D}_n(L)}}
A^{L}_{I,J} \langle f, h_I \rangle h_J,\]
where \(\left\|\tilde{A}^{L}_{I,J}\right\| := \big \| \langle W \rangle^{1/2}_L A^L_{I, J} \langle W \rangle^{-1/2}_L \big \| \leq \frac{\sqrt{|I|} \sqrt{|J|}}{|L|} =2^{-(m+n)p/2} \).

Let \(f \in \ltw,\  g \in \ltwi\) and \(0 \leq t \leq k-1\) be fixed. For the slice \(S_t\), we can write
\begin{align*}
& \left \langle S_t f,g \right \rangle_{\ltw,\ltwi}  = \Bigg \langle \sum_{L \in \mathcal{L}_j}
\sum_{\substack{
            I \in \mathcal{D}_m(L) \\
            J \in \mathcal{D}_n(L)}}
A^{L}_{I,J} \langle f, h_I \rangle h_J, \sum_{I' \in \dd} \langle g, h_{I'} \rangle h_{I'} \Bigg \rangle _{\ltw,\ltwi}\\
& \qquad = \sum_{L \in  \mathcal{L}_t} \sum_{I' \in \dd}
\sum_{\substack{
            I \in \mathcal{D}_m(L) \\
            J \in \mathcal{D}_n(L)}}
 \big \langle A^{L}_{I,J} \langle f, h_I \rangle , \langle g, h_{I'} \rangle \big \rangle _{\mathbb{C}^d} \langle h_J,h_{I'} \rangle_{L^2(\mathbb{R}),L^2(\mathbb{R})} \\
& \qquad = \sum_{L \in  \mathcal{L}_t}
\sum_{\substack{
		I \in \mathcal{D}_m(L) \\
		J \in \mathcal{D}_n(L)}}
\big \langle A^{L}_{I,J} \langle f, h_I \rangle , \langle g, h_J \rangle \big \rangle _{\mathbb{C}^d} \\
& \qquad = \sum_{L \in  \mathcal{L}_t}
\sum_{\substack{
		I \in \mathcal{D}_m(L) \\
		J \in \mathcal{D}_n(L)}}
\big \langle \langle W \rangle_L ^{-1/2} \left(\langle W \rangle_L ^{1/2} A^{L}_{I,J} \langle W \rangle_L ^{-1/2} \right) \langle W \rangle_L ^{1/2} \langle f, h_I \rangle , \langle g, h_J \rangle \big \rangle _{\mathbb{C}^d}  \\
& \qquad = \sum_{L \in  \mathcal{L}_t}
\sum_{\substack{
		I \in \mathcal{D}_m(L) \\
		J \in \mathcal{D}_n(L)}}
\big \langle \langle W \rangle_L ^{-1/2}  \tilde{A}^{L}_{I,J}  \langle W \rangle_L ^{1/2} \langle f, h_I \rangle , \langle g, h_J \rangle \big \rangle _{\mathbb{C}^d}  \\
& \qquad = \sum_{L \in  \mathcal{L}_t}
\sum_{\substack{
		I \in \mathcal{D}_m(L) \\
		J \in \mathcal{D}_n(L)}}
\big \langle \tilde{A}^{L}_{I,J}  \langle W \rangle_L ^{1/2} \langle f, h_I \rangle , \langle W \rangle_L ^{-1/2} \langle g, h_J \rangle \big \rangle _{\mathbb{C}^d}  \\
& \qquad = \sum_{L \in  \mathcal{L}_t} \sum_{i, j = 1}^d  
\sum_{\substack{
            I \in \mathcal{D}_m(L) \\
            J \in \mathcal{D}_n(L)}}
 \big \langle \tilde{A}^{L}_{I,J} \lambda_{L,i}^{1/2} P_L^i \langle f, h_I \rangle , \lambda_{L,j}^{-1/2} P_L^j \langle g, h_J \rangle \big \rangle _{\mathbb{C}^d} \\
 & \qquad = \sum_{L \in  \mathcal{L}_t} \sum_{i, j = 1}^d  
 \sum_{\substack{
 		I \in \mathcal{D}_m(L) \\
 		J \in \mathcal{D}_n(L)}}
 \big \langle \tilde{A}^{L}_{I,J},  \lambda_{L,j}^{-1/2} P_L^j \langle g, h_J \rangle \otimes \lambda_{L,i}^{1/2} P_L^i \langle f, h_I \rangle \big \rangle _{\mathcal{S}_2} \\
& \qquad \leq C \sum_{L \in  \mathcal{L}_t} \sum_{i, j = 1}^d  
  \sum_{\substack{
  		I \in \mathcal{D}_m(L) \\
  		J \in \mathcal{D}_n(L)}}
  \big \|  \lambda_{L,j}^{-1/2} P_L^j \langle g, h_J \rangle \otimes \lambda_{L,i}^{1/2} P_L^i \langle f, h_I \rangle \big \| _{\mathcal{S}_1} \\
  & \qquad = C \sum_{L \in  \mathcal{L}_j} \sum_{i, j = 1}^d  
  \sum_{\substack{
  		I \in \mathcal{D}_m(L) \\
  		J \in \mathcal{D}_n(L)}}
  \big | \big \langle \lambda_{L,j}^{-1/2} e_L^j \otimes  \lambda_{L,i}^{1/2} e_L^i \langle f, h_I \rangle,  \langle g, h_J \rangle \big \rangle _{\mathbb{C}^d} \big|.
\end{align*}

Since \( \langle W \rangle_L ^{1/2} ( \lambda_{L,j}^{-1/2} e_L^j \otimes  \lambda_{L,i}^{1/2} e_L^i ) \langle W \rangle_L ^{-1/2} = e_L^j \otimes e_L^i\), we have 
\[\left\|\left\{\lambda_{L,j}^{-1/2} e_L^j \otimes  \lambda_{L,i}^{1/2} e_L^i \right\}_{L \in \mathcal{D}} \right \|_{\infty, W} = \sup_{L \in \mathcal{D}} \|  \langle W \rangle_L ^{1/2} ( \lambda_{L,j}^{-1/2} e_L^j \otimes  \lambda_{L,i}^{1/2} e_L^i ) \langle W \rangle_L ^{-1/2} \| =  \sup_{L \in \mathcal{D}} \| e_L^j \otimes e_L^i \| = 1. \]
It follows that

\begin{align*}
&\left \langle S_t f,g \right \rangle_{\ltw,\ltwi} \\
& \leq C \sup_{\substack{
		\sigma = \{\sigma_L\}_{L \in \dd}\\
		\|\sigma\|_{\infty, W} \leq 1}}
\sum_{L \in  \mathcal{L}_t} \sum_{i, j = 1}^d  
\sum_{\substack{
		I \in \mathcal{D}_m(L) \\
		J \in \mathcal{D}_n(L)}}
\big | \big \langle \sigma_L \langle f, h_I \rangle,  \langle g, h_J \rangle \big \rangle _{\mathbb{C}^d} \big|\\
& \leq Cd^2  \sup_{\substack{
		\sigma = \{\sigma_L\}_{L \in \dd}\\
		\|\sigma\|_{\infty, W} \leq 1}}
\sum_{L \in  \mathcal{L}_j} 
\sum_{\substack{
            I \in \mathcal{D}_m(L) \\
            J \in \mathcal{D}_n(L)}}
 \frac{|I|^{1/2}}{2^{k-m}} \frac{|J|^{1/2}}{2^{k-n}} \Bigg \langle
\sum_{\substack{
            P \in \mathcal{D}_k(L) \\
            P \subset I^{+}}}
\sigma_L \big ( \langle f \rangle _P -  \langle f \rangle _L \big ) +
\sum_{\substack{
            P \in \mathcal{D}_k(L) \\
            P \subset I^{-}}}
\sigma_L \big ( \langle f \rangle _L -  \langle f \rangle _P \big ) , \Bigg. \\
& \hspace{6 cm} \Bigg.
\sum_{\substack{
            Q \in \mathcal{D}_k(L) \\
            Q \subset J^{+}}}
 \big( \langle g \rangle _Q -  \langle g \rangle _L \big ) +
\sum_{\substack{
            Q \in \mathcal{D}_k(L) \\
            Q \subset J^{-}}}
 \big ( \langle g \rangle _L -  \langle g \rangle _Q \big )
\Bigg \rangle_{\mathbb{C}^d}.
\end{align*}

We therefore have
\begin{align} \label{est:shift}
& \Big | \left \langle S_j f,g \right \rangle_{\ltw,\ltwi}  \Big|   \\ \nonumber
& \leq C \sup_{\substack{
		\sigma = \{\sigma_L\}_{L \in \dd}\\
		\|\sigma\|_{\infty, W} \leq 1}}
\sum_{L \in  \mathcal{L}_j} |L|
\sum_{P,Q \in \mathcal{D}_k(L)} \bigg | \bigg \langle \sigma_L \bigg (  \frac{\langle f \rangle _P -  \langle f \rangle _L}{2^k} \bigg ),  \bigg (\frac{\langle g \rangle _Q -  \langle g \rangle _L}{2^k}  \bigg )  \bigg \rangle_{\mathbb{C}^d}  \bigg|.  \nonumber
%& \le \sup_{ \|\sigma\|_{\infty, W} \le 1}  d \cdot \left|  \sum_{L \in  \mathcal{L}_j} |L| 
%\sum_{P,Q \in \mathcal{D}_k(L)} \bigg \langle \sigma_L \bigg (  \frac{\langle f \rangle _P -  \langle f \rangle _L}{2^k} \bigg ),  \bigg (\frac{\langle g \rangle _Q -  \langle g \rangle _L}{2^k}  \bigg )  \bigg \rangle_{\mathbb{C}^d}  \right|  .
\end{align}

Using the definition of the martingale transform operator \(T_{\sigma}\), we can write
\begin{align} \label{est:projections}
& \sup_{\substack{
		\sigma = \{\sigma_I\}_{I \in \dd}\\
		\|\sigma\|_{\infty, W} \leq 1}}
\sd \left | \big \langle  \sigma_I \big( \langle f \rangle _{I^+} - \langle f \rangle_{I^-} \big), \big( \langle g \rangle _{I^+} - \langle g\rangle_{I^-} \big) \big \rangle_{\mathbb{C}^d} \right | \cdot |I| \nonumber \\
& 4 \sup_{\substack{
		\sigma = \{\sigma_I\}_{I \in \dd}\\
		\|\sigma\|_{\infty, W} \leq 1}}
\sd \left| \big \langle \sigma_I \langle f,h_I \rangle, \langle g,h_I \rangle \big \rangle_{\mathbb{C}^d} \right| \\   \nonumber
& = 4 \sup_{\substack{
		\sigma = \{\sigma_I\}_{I \in \dd}\\
		\|\sigma\|_{\infty, W} \leq 1}}
\left| \langle T_\sigma f,g \rangle_{\ltw,\ltwi} \right| \\\nonumber
&\leq 4 \sup_{\substack{
		\sigma = \{\sigma_I\}_{I \in \dd}\\
		\|\sigma\|_{\infty, W} \leq 1}}
\|T_{\sigma}\|_{\ltw \to \ltw}  \|f \|_{\ltw} \|g\|_{\ltwi} \\   \nonumber
&\leq 4 N([W]_{A_2^d}) \|f \|_{\ltw} \|g\|_{\ltwi}. \nonumber
\end{align}

The left hand side of this chain of inequalities is what motivates the following definition of the Bellman function associated to our problem. Notice that this Bellman function is exactly the same as the one that appears in \cite{PoSt15}.

Let \(X >1\), fix a dyadic interval \(I_0\) and for \(\fd \in \mathbb{C}^d, \Fd \in \mathbb{R}, \ud \in \mathcal{M}_d(\mathbb{C}), \gd \in \mathbb{C}^d, \Gd \in \mathbb{R}, \vd \in \mathcal{M}_d(\mathbb{C})\) satisfying
\begin{equation}\label{domain}
\ud, \vd >0, I_d \leq \vd^{1/2} \ud \vd^{1/2} \leq X \cdot I_d, \|\vd^{-1/2} \fd\|_{\mathbb{C}^d}^2 \leq \Fd,\ \| \ud^{-1/2} \gd\|_{\mathbb{C}^d}^2 \leq \Gd ,
\end{equation}
define the function \(\fb_X=\fb_X^{I_0}: \mathbb{C}^d \times \mathbb{R} \times \mathcal{M}_d(\mathbb{C}) \times \mathbb{C}^d \times \mathbb{R} \times \mathcal{M}_d(\mathbb{C}),\) by
\begin{equation*}
\fb_X(\fd, \Fd, \ud, \gd, \Gd, \vd):=
 |I_0|^{-1} \sup \sum_{I \subseteq I_0} \left | \big \langle \sigma_I \big ( \langle f \rangle _{I^+} - \langle f \rangle_{I^-} \big ), \langle g \rangle _{I^+} - \langle g\rangle_{I^-}   \big \rangle_{\mathbb{C}^d} \right | \cdot |I|,
%=  |I_0|^{-1} \sup \left| \sum_{I \subseteq I_0}   \langle \sigma_I \big ( \langle f \rangle _{I^+} - \langle f \rangle_{I^-} \big ), \langle g \rangle _{I^+} - \langle g\rangle_{I^-} \big \rangle_{\mathbb{C}^d}  \cdot |I| \right| , 
 \end{equation*}
where the supremum is taken over all functions \(f,g :\mathbb{R} \to \mathbb{C}^d,\) all matrix  \(A_2^d\) weights \(W\) on \(I_0,\) and all sequences of \(d \times d\) matrices \(\sigma = \{\sigma_I\}_{I \in \mathcal{D}}\) such that
\begin{equation}\label{supdomain}
\langle f \rangle _{I_0}=\fd \in \mathbb{C}^d, \quad \big \langle \|W^{1/2} f\|^2_{\mathbb{C}^d} \big \rangle_{I_0} = \Fd \in \mathbb{R}, \quad \langle g \rangle _{I_0}=\gd \in \mathbb{C}^d, \quad \big \langle \|W^{-1/2} g\|^2_{\mathbb{C}^d} \big \rangle_{I_0} = \Gd \in \mathbb{R}, 
\end{equation}
\begin{equation}\label{weightsigmadomain}
\sup_{\substack {
I \in \mathcal{D} \\
I \subset I_0} }
\|\langle W \rangle_I ^{1/2} \langle W^{-1} \rangle_I ^{1/2} \|^2 \leq X, \quad \langle W \rangle _{I_0} = \ud, \quad \langle W^{-1} \rangle_{I_0} = \vd, \quad \|\sigma\|_{\infty, W} \leq 1.
\end{equation}

The Bellman function \(\fb_X\) has the following properties:
\begin{enumerate}[(i)]
    \item (Domain) The domain \(\mathfrak{D}_X:=\mathrm{Dom}\, \fb_X\) is given by \eqref{domain}. This means that for every tuple \((\fd, \Fd, \ud, \gd, \Gd, \vd)\) that satisfies \eqref{domain}, there exist functions \(f, g\) and a matrix weight \(W\) such that \eqref{supdomain} holds, so the supremum is not \(-\infty\). Conversely, if the variables \(\fd, \Fd, \ud, \gd, \Gd, \vd\) are the corresponding averages of some functions \(f, g\) and \(W\), then they must satisfy condition \eqref{domain}. Since the set \(\{(\ud,\vd) \in \mathcal{M}_d(\mathbb{C}) \times \mathcal{M}_d(\mathbb{C}): \ud,\vd >0, I_d \leq \vd^{1/2} \ud \vd^{1/2} \leq X \cdot I_d\}\) is not convex, the domain \(\mathfrak{D}_X\) is not convex either. 
   \item (Range) \(0 \leq \fb_X(\fd, \Fd, \ud, \gd, \Gd, \vd) \leq 4  N(X) \Fd^{1/2} \Gd^{1/2}\) for all \((\fd, \Fd, \ud, \gd, \Gd, \vd) \in \mathfrak{D}_X.\)
    \item (Concavity condition) Consider all tuples \(A=(\fd, \Fd, \ud, \gd, \Gd, \vd), A_+=(\fd_+, \Fd_+, \ud_+, \gd_+, \Gd_+, \vd_+)\) and \(A_-=(\fd_-, \Fd_-, \ud_-, \gd_-, \Gd_-, \vd_-)\) in \(\mathfrak{D}_X\) such that \(A=(A_+ + A_-)/2\). 
    For all such tuples, we have the following concavity condition:
        \[\fb_X(A) \geq \frac{\fb_X(A_+)+\fb_X(A_-)}{2} + 
             \sup_{\| \tau\|_{\ud} \le 1 } \left| \left \langle  \tau (\fd_+ - \fd_- ), \gd_+ - \gd_-  \right \rangle_{\mathbb{C}^d}  \right| .\]
  \end{enumerate}
Here, the supremum is taken over all $d \times d$ matrices $\tau$ with $\|\tau\|_{\ud} := \|\ud^{1/2} \tau \ud^{-1/2}\| \le 1$.

More details about these properties can be found in \cite{PoSt15}.

We can now state the main tool for the proof of Theorem \ref{mainshift}.

\begin{lemma}\label{mainlemma}
Let \(X>1\) and \(\fb_X\) be a function satisfying properties (i)-(iii) from above. Fix \(k \geq 1\) and a dyadic interval \(I_0\). For all \(I \in \mathcal{D}_n(I_0),\ 0 \leq n \leq k,\) let the points \(A_I= (\fd_I, \Fd_I, \ud_I, \gd_I, \Gd_I, \vd_I) \in \mathfrak{D}_X=\mathrm{Dom}\, \fb_X\) be given. Assume that the points \(A_I\) satisfy the dyadic martingale dynamics, i.e. \(A=(A_{I^+}+A_{I^-})/2,\) where \(I^+\) and \(I^-\) are the children of \(I\).  Let \(\sigma_{I_0}\) be a \(d \times d\) matrix such that \( \|\ud_{I_0}^{1/2} \sigma_{I_0} \ud_{I_0}^{-1/2}\| \leq 1\).
For \(K,L \in \mathcal{D}_k(I_0)\), we define the coefficients \(\lambda_{KL}\) by
\[\lambda_{KL} := \bigg \langle \sigma_{I_0} \bigg ( \frac{\fd_K - \fd_{I_0}}{2^k} \bigg ),  \bigg ( \frac{\gd_L - \gd_{I_0}}{2^k} \bigg ) \bigg \rangle_{\mathbb{C}^d}  .\]

Then
\[ \sum_{K,L \in \mathcal{D}_k({I_0})} |\lambda_{KL}| \leq c \cdot d \bigg( \fb_{X'}(A_{I_0}) - 2^{-k} \sum_{I \in \mathcal{D}_k(I_0)} \fb_{X'}(A_I) \bigg),\]
where \(c\) is a positive absolute constant and \(X'=\frac{100}{9}X\).
\end{lemma}

The proof of this lemma follows exactly as in \cite{PoSt15}, the only difference being the use of the matrix \(\sigma_{I_0}\) instead of the projection \(P_{I_0}\).

We are now ready to finish the proof of Theorem \ref{mainshift}.

Recall that for all slices \(S_t\) of \(S\) we have
\[ \Big | \left \langle S_t f,g \right \rangle_{\ltw,\ltwi} \Big|  \leq C d^2 \sup_{\substack{
		\sigma = \{\sigma_I\}_{I \in \dd}\\
		\|\sigma\|_{\infty, W} \leq 1}}
\sum_{L \in  \mathcal{L}_t} |L| 
\sum_{P,Q \in \mathcal{D}_k(L)} \bigg | \bigg \langle \sigma_L \bigg (  \frac{\langle f \rangle _P -  \langle f \rangle _L}{2^k} \bigg ),  \frac{\langle g \rangle _Q -  \langle g \rangle _L}{2^k}    \bigg \rangle_{\mathbb{C}^d}  \bigg|. \]

Let \(X:= [W]_{A_2^d}\); fix \(0 \leq t \leq k-1 \) and for all \(I \in \mathcal{L}_t\) define
\[A_I:=\Big (\langle f \rangle _I, \big \langle \|W^{1/2}f\|^2_{\mathbb{C}^d}  \big \rangle_I, \langle W \rangle _I, \langle g \rangle _I, \big \langle \|W^{-1/2}g\|^2_{\mathbb{C}^d} \big \rangle_I, \langle W^{-1} \rangle _I \Big).\]
Notice that all these points are in \(\mathrm{Dom}\, \fb_X=\mathfrak{D}_X\). Lemma \ref{mainlemma} says that
\begin{align*}
& |L| 
\sum_{P,Q \in \mathcal{D}_k(L)} \bigg | \bigg \langle \sigma_L \bigg (  \frac{\langle f \rangle _P -  \langle f \rangle _L}{2^k} \bigg ), \frac{\langle g \rangle _Q -  \langle g \rangle _L}{2^k}   \bigg \rangle_{\mathbb{C}^d}  \bigg| \\
& \qquad \qquad \leq c \cdot d \bigg( |L| \fb_{X'}(A_L) - \sum_{I \in \mathcal{D}_k(L)} |I| \fb_{X'}(A_I) \bigg ),
\end{align*}
for all \(L \in \mathcal{L}_t\) and all \(d \times d\) matrices \(\sigma_L\) such that \(\| \langle W \rangle _L^{1/2} \sigma_L \langle W \rangle _L^{-1/2}\| \leq 1\).  We write this estimate for each \(I \in \mathcal{D}_k(L)\) and then iterate the procedure \(\ell\) times to obtain
\begin{align*}
&  \sum_{\substack{
                    I \in \mathcal{L}_t \\
                    I \subseteq L\\
                    |I| > 2^{-k \ell}|L| }}
|I|  \sum_{P,Q \in \mathcal{D}_k(I)} \bigg | \bigg \langle  \sigma_L \bigg ( \frac{\langle f \rangle _P -  \langle f \rangle _I}{2^k} \bigg), \frac{\langle g \rangle _Q -  \langle g \rangle _I}{2^k}     \bigg \rangle_{\mathbb{C}^d} \bigg|  \\
& \qquad \qquad \qquad \leq c \cdot d  \bigg( |L| \fb_{X'}(A_L) - \sum_{I \in \mathcal{D}_{k \ell}(L)} |I| \fb_{X'}(A_I) \bigg ) \\
& \qquad \qquad \qquad \leq c \cdot d N(X') |L| \big \langle \|W^{1/2}f\|^2_{\mathbb{C}^d}  \big \rangle_L^{1/2} \big \langle \|W^{-1/2}g\|^2_{\mathbb{C}^d}  \big \rangle_L^{1/2} \\
& \qquad \qquad \qquad \leq c \cdot d N(X') \|f \chi_L\|_{\ltw}  \|g \chi_L\|_{\ltwi},
\end{align*}
where the second inequality follows from property (ii) of the Bellman function.
\par Letting \(\ell \to \infty\), we have
\begin{align*}
&  \sum_{\substack{
                    I \in \mathcal{L}_t \\
                    I \subseteq L }}
|I| \sum_{P,Q \in \mathcal{D}_k(I)} \bigg | \bigg \langle \sigma_L \bigg ( \frac{\langle f \rangle _P -  \langle f \rangle _I}{2^k} \bigg), \ \frac{\langle g \rangle _Q -  \langle g \rangle _I}{2^k}     \bigg \rangle_{\mathbb{C}^d} \bigg|  \\
& \qquad \qquad \qquad \leq c \cdot d N(X') \|f \chi_L\|_{\ltw}  \|g \chi_L\|_{\ltwi}.
\end{align*}

We now cover the real line with intervals \(L \in \mathcal{L}_t\) of length \(2^M\) and apply the last inequality to each \(L\) to obtain that
\begin{align*}
&  \sup_{\substack{
		\sigma = \{\sigma_I\}_{I \in \dd}\\
		\|\sigma\|_{\infty, W} \leq 1}}
 \sum_{\substack{
                    I \in \mathcal{L}_t \\
                    |I| \leq 2^M }}
|I| \sum_{P,Q \in \mathcal{D}_k(I)} \bigg | \bigg \langle  \sigma_L \bigg ( \frac{\langle f \rangle _P -  \langle f \rangle _I}{2^k} \bigg),  \frac{\langle g \rangle _Q -  \langle g \rangle _I}{2^k}     \bigg \rangle_{\mathbb{C}^d} \bigg|  \\
& \qquad \qquad \qquad \leq c \cdot d N(X) \|f\|_{\ltw}  \|g\|_{\ltwi}.
\end{align*}

When \(M \to \infty\), we get that the norm of \(S_t\) is bounded by \(c \cdot d^3 N(X)\). Since \(S\) was decomposed into \(k\) slices, 
it follows that the operator norm of \(S\) is bounded by \(c \cdot k d^3  N([W]_{A_2^d})\), and therefore the proof of Theorem \ref{mainshift} (and therefore also of the sufficiency in Theorem \ref{mainthm}) is complete.
\qed

\bibliographystyle{plain}

\end{document}